\documentclass[EJP,preprint]{ejpecp}

\usepackage{enumerate}
\usepackage{wrapfig}

\author{Xinyi Li \thanks{Beijing International Center for Mathematical Research, Peking University}  \and  Daisuke Shiraishi \thanks{Graduate School of Informatics, Kyoto University }}
\date{\today}

\SHORTTITLE{H\"older continuity of scaling limit of 3D LERW}

\TITLE{The H\"older continuity of the scaling limit of three-dimensional loop-erased random walk} 

\AUTHORS{%
  Xinyi~Li\footnote{Beijing International Center for Mathematical Research, Peking University, Beijing, China.\newline
    \EMAIL{xinyili@bicmr.pku.edu.cn}}
  \and 
  Daisuke~Shiraishi\footnote{Graduate School of Informatics, Kyoto University, Kyoto, Japan. \EMAIL{shiraishi@acs.i.kyoto-u.ac.jp} }}


\AMSSUBJ{82B41; 60G18} 

\SUBMITTED{November 12, 2021} 
\ACCEPTED{October 19, 2022} 

\ARXIVID{2111.04977} 

\VOLUME{0}
\YEAR{2020}
\PAPERNUM{0}
\DOI{10.1214/YY-TN}

\ABSTRACT{Let $\beta$ be the growth exponent of the loop-erased random walk (LERW) in three dimensions. We prove that the scaling limit of 3D LERW is almost surely $h$-H\"older continuous for all $h < 1/\beta$, but not $1/\beta$-H\"older continuous.
}

\newtheorem{dfn}{Definition}[section]
\newtheorem{thm}[dfn]{Theorem}

\newtheorem{lem}[dfn]{Lemma}

\newtheorem{rem}[dfn]{Remark}
\newtheorem{prop}[dfn]{Proposition}

\newcommand{\LE}{{\rm LE}}
\newcommand{\Es}{{\rm Es}}

\begin{document}



\section{Introduction}\label{sec:1}

Loop-erased random walk (LERW) was first introduced by Greg Lawler in \cite{Lawler}. It is obtained by erasing loops from the path of simple random walk (SRW) chronologically. Arguably 
LERW is the most tractable of all non-Gaussian models appearing in statistical physics. It has been proved that the scaling limit of LERW is a Brownian motion above four dimensions (\cite{Lawb}). On the other hand, the scaling limit is described by Schramm-Loewner 
evolution with parameter $2$ ($\text{SLE}_{2}$) in two dimensions (\cite{Sch}, \cite{LSW}, \cite{LV}).

The three-dimensional case, on the contrary, remains poorly understood. Remarkably, Kozma (\cite{Koz}) proved that the scaling limit of LERW in three dimensions is well-defined as a set. In stark contrast to the two-dimensional case, his proof does not identify what the scaling limit is. Starting from various fine properties of LERW and Kozma's scaling limit result (\cite{S}, \cite{SS}, \cite{S2}, \cite{Escape}), the authors showed (\cite{LS}) that LERW converges weakly as a process with time parametrization, which greatly strengthens Kozma's result. This improvement plays a key role in \cite{ACHS} to prove that the three-dimensional uniform spanning tree (UST) converges weakly as a metric space endowed with the graph distance.

The aim of this article is to study the H\"older continuity of the time parametrized version of the scaling limit in three dimensions. 
Let $\eta $ be the scaling limit of the three-dimensional LERW on rescaled lattice in the unit ball started from the origin. Write $\beta \in (1, 5/3]$  for the growth exponent for 3D LERW as well as the a.s.\ Hausdorff dimension of $\eta$. See Section \ref{LERW-intro} (in particular \eqref{betadef} and Section \ref{SCALING} for the definition of  $\beta$ and $\eta$ resp.) for precise definitions of notation.

Our main result is the following theorem.

\begin{thm}\label{main}
It holds that $\eta$ is a.s.\ $\beta$-H\"older continuous for all $h <1/\beta$, but not $1/\beta$-H\"older continuous.
\end{thm}

\begin{rem}\label{mainrmk}
One may compare Theorem \ref{main} with the H\"older continuity of the scaling limit of LERW in other dimensions. In two dimensions, as LERW converges to Schramm-Loewner evolution (SLE) with parameter $2$ (see \cite{Sch,LSW,LV}), the study of the same question is covered by that of the SLE curves. The optimal H\"older exponent of ${\rm SLE}_2$ under natural parametrization (i.e. parametrized through Minkowski content), which best resembles our study in 3D here, is studied in \cite{Zhan}, and it is proved that ${\rm SLE}_2$ is $\alpha$-H\"older continuous for any $\alpha<4/5$, but not $4/5$-H\"older continuous. (Recall that ${\rm SLE}_2$ is a $5/4$-dimensional random curve).  Note that the optimal H\"older exponent for SLE under capacity parametrization (see \cite{LV2}), which is a different quantity, is less relevant here as capacity parametrization does not capture the natural convergence of 2D LERW towards ${\rm SLE}_2$.  In four and higher dimensions, the question is less interesting since the scaling limit of LERW is Brownian motion, whose H\"older continuity is well known.

It is also an interesting question whether an analogue of McKean's dimension theorem  for Brownian motion and ${\rm SLE}_\kappa$ also holds in this case (cf.\ Theorem 1.3 of \cite{Zhan}): is it true that for each deterministic closed set $A\subset \mathbb{R}^+$,
\begin{equation}
    {\rm dim_H}(\eta(A)) = \beta  {\rm dim_H}(A)?
\end{equation}
Here ${\rm dim_H}(\cdot)$ stands for the Hausdorff dimension of a set.
\end{rem}

We now give a few words about the idea of the proof and the structure of the paper.
To prove the $h$-H\"older continuity for $h<1/\beta$, the key ingredient is the exponential tail bound on the length of 3D LERW from \cite{S} (see \eqref{exp-tail}). However, it is only available for the beginning part of $\eta$. To circumvent this, we resort to the uniform spanning tree through Wilson's algorithm and apply this tail bound to a dense deterministic ``net'', such that with high probability the beginning part of LERW from each point in this net covers the whole of $\eta$. See Section \ref{sec:3.2}, in particular the proof of Proposition \ref{HC-upper} for more details. 

For the second claim, we will first consider the event that LERW keeps moving with no big backtracking in a small-bore ``tube'' in which the LERW substantially speeds up and we have a good control on the modulus of continuity.
After we give a lower bound on the probability of this event,  we will carry out an iteration argument to make this event happen with high probability. See the beginning of Section \ref{sec:4} for a more detailed explanation of the proof.

Before we finish this section, let us explain the organization of this paper here. General notation and background on LERW will be introduced in Section \ref{sec:2}. Then the two claims of Theorem \ref{main} will be treated in Sections \ref{sec:3} and \ref{sec:4} respectively. The proof of Theorem \ref{main} follows from Propositions \ref{0626-a-1} and \ref{CRITICAL}. 
%
%

\section{Notation and preliminaries}\label{sec:2}
In this section, we will introduce notation and some preliminary results (mainly on loop-erased random walk) which will be useful later.

\subsection{General notation}\label{se:general}
Let $\mathbb{R}^d$ and $\mathbb{Z}^d$ stand for the $d$-dimensional Euclidean space and integer lattice respectively.  We will consider only $d=3$ in this paper. For $a>0$, let $a\mathbb{Z}^d$ stand for the rescaled lattice with mesh size $a$. For most of the time, we focus on the case $a=2^{-n}$ for some integer $n \ge 0$.

Given a subset $A \subset \mathbb{R}^{d}$, we write $\partial A$ for its boundary. Set $\overline{A} = A \cup \partial A$ 
for its closure. Let $| \cdot |$ and $\| \cdot \|_\infty$ denote the Euclidean and maximum norm in $\mathbb{R}^{d}$ respectively. For $A, B \subset \mathbb{R}^{d}$, we let ${\rm dist} (A, B) = \inf_{x \in A, y \in B} |x-y|$ be the distance between $A$ and $B$. When $A = \{ x \}$, we set ${\rm dist} (\{ x \}, B) = {\rm dist} (x, B)$. Let $A+B = \{ a + b \ | \ a \in A, \ b \in B \}$. When $A = \{ x \}$,  we write $x + B $ instead of $\{ x \} + B$. We also write $r A = \{ r a \ | \ a \in A \}$ for $r > 0$. Let $B (x, r)$ and $B (r)$ stand for $\{ y \in \mathbb{R}^{d} \ | \ |x- y| < r \}$ and $\{ y \in \mathbb{R}^{d} \ | \ |y| < r \}$ resp. In particular, write $\mathbb{D} =B(1)$ for the unit open ball in $\mathbb{R}^{d}$.

Given a subset $A \subset  2^{-n}\mathbb{Z}^{d}$, we write $\partial A = \{ x \in  2^{-n}\mathbb{Z}^{d} \setminus A \ | \ \exists y \in A \text{ such that } |x-y| = 2^{-n}  \}$ for the outer boundary of $A$. 
Take $x \in 2^{-n}\mathbb{Z}^{d}$ and $r > 0$. The discrete ball of radius $r$ centered at $x$ is denoted by $B^{(n)} (x, r) = \{ y \in 2^{-n}\mathbb{Z}^{d} \ | \ |x-y| < r \}$. We write $B^{(n)} (r) = B^{(n)} (0, r)$. 

\medskip

A sequence of points $\lambda = [\lambda (0), \lambda (1), \cdots , \lambda (m) ] \subset 2^{-n}\mathbb{Z}^{d}$   lying on $2^{-n}\mathbb{Z}^{d}$ is called a path with length $m$ if $| \lambda (i-1)  - \lambda (i) | = 2^{-n}$ for each $1 \le i \le m$. 
We denote the length by ${\rm len} (\lambda)$. We say $\lambda$ is a simple path if $\lambda (i) \neq \lambda (j) $ for all $i \neq j$. When we have two paths $\lambda = [\lambda (0), \lambda (1), \cdots , \lambda (m) ] \subset 2^{-n}\mathbb{Z}^{d}$ and $\lambda' = [\lambda' (0), \lambda' (1), \cdots , \lambda' (m') ] \subset 2^{-n}\mathbb{Z}^{d}$  satisfying $\lambda (m) = \lambda' (0)$, we write 
\begin{equation*}
\lambda \oplus \lambda' = [\lambda (0), \lambda (1), \cdots , \lambda (m), \lambda' (1), \cdots , \lambda' (m') ]
\end{equation*}
for their  concatenation.

Before we end this subsection, we state our conventions on constants and asymptotics.

When $f (x)$ and $g (x)$ are two positive functions, we write  
 \begin{equation*}
 f(x) \asymp g(x)
 \end{equation*}
 if there exists a universal constant $c > 0$ such that $c f(x) \le g(x) \le \frac{1}{c} f(x)$ for every $x>0$.
We also write
\begin{equation*}
g(x) = O \big( f(x) \big) \ \ \ \text{ and }  \ \ \  g (x) = o \big( f(x) \big)
\end{equation*}
if 
\begin{equation*}
|g (x)| \le C f(x)  \ \ \ \text{ and }  \ \ \  g (x ) / f (x) \to 0,
\end{equation*}
respectively.

We use $c, c', C, C',\ldots$ to denote positive constants which may change from place to place and $c_{i}, C_{i}, c_{\star},\ldots$ positive constants that are kept fixed throughout the paper. Dependence on other variables or constants will be marked at the first appearance of the constants. 
\subsection{Simple random walk and Green's functions}\label{sec:SRW}
In this subsection, we will introduce the Green's function for simple random walk and list some estimates on return probabilities  that we will use later.

Let $S = S^{(n)} = \big( S^{(n)} (k)  \big)_{k \ge 0}$  be  a simple random walk (SRW) in $2^{-n}\mathbb{Z}^{d}$. The superscript $(n)$ is omitted whenever this does not cause any confusion. To denote its probability law and expectation, we use $P^{x}$ and $E^{x}$ when the SRW starts from $x$. When $x=0$, we omit the superscript. When there is need to work with multiple random walks, we will use other letters to represent other walks and subscripts to distinguish their respective laws; e.g.,  $P^x_X$ is the law of the simple random walk $X$, etc.

Given $A\subset2^{-n}\mathbb{Z}^3$, let $\tau_A = \inf \{ k \ge 0 \ | \ S(k) \notin  A \}$ be the first time $S$ exits from $A$ and write $T_{x, r}= \tau_{B^{(n)} (x, r)}$ for the exit time of the discrete ball with the convention that $\inf \emptyset = + \infty $. We also write $T_{r} = T_{0, r}$. If we are working with multiple walks at the same time, we will use superscripts to distinguish them; e.g., $T^X_{x,r}$, $T^{X}_{r}$ and $\tau^X_A$ stand for the first exit time of $X$ from $B^{(n)} (x, r)$, $B^{(n)} (r)$ and some set $A$ respectively.

We now turn to the Green's function and some of its properties. The Green's function $G :2^{-n} \mathbb{Z}^{3} \times 2^{-n}\mathbb{Z}^{3} \to [0, \infty )$ for $S$ is defined as
\begin{equation}\label{green-entire}
G (x, y) = E^{x} \bigg( \sum_{k =0}^{\infty} {\bf 1} \{ S (k) = y \} \bigg)  
\end{equation}
$\text{ for } x, y \in \mathbb{Z}^{3}.$ We write $G (x)  = G (0, x)$ for short.

Given a set $A \subset \mathbb{Z}^{d}$, the Green's function (for $S$) on $A$ is defined by 
\begin{equation}\label{grin}
G_{A} (x, y) = E^{x} \bigg( \sum_{k = 0}^{\tau_A} {\bf 1} \{ S(k) = y \} \bigg) \ \ \text{ for } \ x,y \in A.
\end{equation}

For the case that $n=0$ and $A = \{ x \in \mathbb{Z}^{3} \ | \ a \le |x| \le b \}$ with $1 \le a < b $, it follows from \cite[Proposition 1.5.10]{Lawb} that 
\begin{equation}\label{srwbound}
P^{x} \big( |S_{\tau} | \le a \big) = \frac{ |x|^{-1} - b^{-1}  + O (a^{-2} )  }{a^{-1} - b^{-1}}.
\end{equation}
for all $x \in A$.
When $a=1$ and $|x|\leq b$, an easy modification shows that there exists a universal constant $c_{0} > 0$ such that 
\begin{equation}\label{srwbound-2}
P^{x} \big( S_{\tau} = 0 \big) = \frac{ c_{0}  |x|^{-1} - c_{0}  b^{-1}  + O (|x|^{-2} )  }{G (0) - c_{0}  b^{-1}},
\end{equation}
see (2.4) in \cite{LS} for more details.

\subsection{Some estimates on SRW}

In this subsection, we will review the Harnack principle and the gambler's ruin estimate. 

Let $O \subset \mathbb{R}^{d}$ be a connected open set and $F $ be a compact subset of $O$. Then the Harnack principle states that there exist $C = C (F, O) < \infty$ and $N = N (F, O) < \infty$ such that if $n \ge N$,
\begin{equation*}
O_{n} = \{ x \in \mathbb{Z}^{d} \ | \ n^{-1} x \in O \},  \ \ \ \ \ \  F_{n} = \{ x \in \mathbb{Z}^{d} \ | \ n^{-1} x \in F \}
\end{equation*}
and $f : O_{n} \to [0, \infty )$ is discrete harmonic in $O_{n}$, then  
\begin{equation}\label{Harnack}
f (x) \le C f (y) 
\end{equation}
for $x, y \in F_{n}$. See \cite[Theorem 6.3.9]{LawLim} for further details. 

We will next recall the gambler's ruin estimate in the following form. Let $S$ be the SRW in $\mathbb{Z}^{d}$ ($d \ge 2$) started at $x \in \mathbb{Z}^{d}$ and take $\theta \in \mathbb{R}^{d}$ with $| \theta | =1$ and $x \cdot \theta > 0$, where $x \cdot \theta$ stands for the usual inner product of $x$ and $\theta$ on $\mathbb{R}^{d}$.  We write $\widetilde{S} (k) = S (k) \cdot \theta$.  Then the gambler's ruin estimate ensures that there exist universal constants $0 < c_{1}, c_{2} < \infty$ (which do not depend on $\theta$ and $x$) such that if $r \ge 1$, $t_{r} = \inf \{ k \ge 0 \ | \ \widetilde{S} (k) \le 0 \text{ or } \widetilde{S} (k) \ge r \}$ and $0 \le x \cdot \theta \le r$, then 
\begin{equation}\label{Gambler}
c_{1} \frac{x \cdot \theta +1}{r} \le P^{x} \big( \widetilde{S} ( t_{r} ) \ge r \big) \le c_{2} \frac{x \cdot \theta +1}{r}.
\end{equation}
See \cite[Proposition 5.1.6]{LawLim} for more details.

\subsection{Metric spaces}\label{metric}
Let ${\cal C} (\overline{\mathbb{D}} )$ denote the space of continuous curves $\lambda : [0, t_{\lambda}] \to \overline{\mathbb{D}} $ with finite time duration $t_{\lambda} \in [0, \infty)$. 
Given  two continuous curves $\lambda_{i} \in {\cal C} (\overline{\mathbb{D}} )$ ($i=1,2$) where $t_{\lambda_{i}} \in [0, \infty)$ stands for the time duration of $\lambda_{i}$,  we then define $\rho (\lambda_{1}, \lambda_{2} )$ by 
\begin{equation}\label{rho-metric}
\rho (\lambda_{1}, \lambda_{2} ) = |t_{\lambda_{1}} - t_{\lambda_{2}} | + \max_{0 \le s \le 1} \big| \lambda_{1} ( s t_{\lambda_{1}} ) - \lambda_{2} ( s t_{\lambda_{2}} ) \big|.
\end{equation}
It is routine to check that $\big( {\cal C} (\overline{\mathbb{D}} ), \rho \big)$ is a metric space (see Section 2.4 of \cite{KS} for this).

We next set $\big( {\cal H} (\overline{\mathbb{D}} ), d_{\text{Haus}} \big)$ for the space of all non-empty compact subsets of $\overline{\mathbb{D}}$, where the Hausdorff distance  $d_{\text{Haus}}$ is defined by 
\begin{equation}\label{Haus}
d_{\text{Haus}} (A, B) := \max \Big\{ \sup_{x \in A} \inf_{y \in B} |x-y|,  \  \sup_{y \in B} \inf_{x \in A} |x-y| \Big\}, 
\end{equation}
for $ A, B \in {\cal H} (\overline{\mathbb{D}})$.
Then $\big( {\cal H} (\overline{\mathbb{D}} ), d_{\text{Haus}} \big)$ becomes a metric space (see Proposition 7.3.3 of \cite{BBI} for this).

\subsection{Loop-erased random walk}\label{LERW-intro}
\subsubsection{Loop erasure}

For a path $\lambda = [\lambda (0), \lambda (1), \cdots , \lambda (m) ] \subset 2^{-n}\mathbb{Z}^{d}$, we define the loop-erasure  $\text{LE} (\lambda )$ as follows. Let 
\begin{equation*}
s_{0} = \max \{ j \ | \ \lambda (j) = \lambda (0) \}
\end{equation*}
and 
\begin{equation*}
s_{i} = \max \{ j \ | \ \lambda (j) = \lambda ( s_{i-1} +  1 ) \} \ \ \text{ for } i \ge 1.
\end{equation*}
Setting
\begin{equation*}
k = \min \{ i \ | \ \lambda ( s_{i} ) = \lambda (m) \},
\end{equation*}
the loop erasure $\text{LE} (\lambda )$ of $\lambda$ is given by 
\begin{equation}\label{LEP}
\text{LE} (\lambda) = [\lambda  (s_{0}), \lambda (s_{1} ), \cdots , \lambda  (s_{k} ) ],
\end{equation} 
which is a simple path satisfying that $\text{LE} (\lambda) \subset \lambda$, $\text{LE} (\lambda) (0) = \lambda (0)$ and $\text{LE} (\lambda) (k) = \lambda (m) $.

Recall that $S= S^{(n)}$ stands for the SRW on $2^{-n}\mathbb{Z}^{d}$. Although we will refer to $\text{LE} \big( S[0, T] \big) $ as a loop-erased random walk (LERW) for any (random or non-random) time $T$, in this work, we will always work with the following setup when we consider the H\"older continuity: 
\begin{equation}\label{gammandef}
    \gamma_n:=\text{LE}\Big(S\big[0,T_1\big]\Big).
\end{equation}
We denote its law and respective expectation by $P$ and $E$ respectively.

\subsubsection{Domain Markov property}\label{DMPLEW}
The LERW is not a Markov process. However, it satisfies the domain Markov property in the following sense.  

Take a finite set $A \subset 2^{-n} \mathbb{Z}^{d}$ and a point $x \in A$  
and write (following the notation in Section \ref{sec:SRW}) $\gamma= \text{LE} \big( S [0, \tau_{A} ] \big)$. Take two simple paths $\lambda = [\lambda (0), \lambda (1), \cdots , \lambda (m)]$ and  $\lambda' = [\lambda' (0), \lambda' (1), \cdots , \lambda' (m')]$ with $\lambda (m) = \lambda' (0)$ satisfying 
\begin{equation*}
P \big( \gamma = \lambda \oplus \lambda' \big) > 0.
\end{equation*}
Let $X$ be the random walk on $2^{-n} \mathbb{Z}^{d}$ started at $\lambda (m)$ conditioned that $X [1, \tau^{X}_{A}]  \cap \lambda = \emptyset$. 
Then the domain Markov property gives the law of $\gamma [m, m+m']$ conditioned that $\gamma [0, m] = \lambda$ via the loop-erasure of $X$ in the following way (see \cite[Proposition 7.3.1]{Lawb} for this).

\begin{equation}\label{DMP}
P \Big( \gamma [0, m+ m']  = (\lambda \oplus \lambda' )  [0, m+ m'] \ \Big| \ \gamma [0, m] = \lambda \Big) = P \Big( \text{LE} \big( X [0, \tau^{X}_{A} ] \big) = \lambda' \Big).
\end{equation}

\subsubsection{Length of LERW in $d=3$}\label{LENGTH}
We will  consider only the $d=3$ case here and in Sections \ref{ONE}, \ref{SCALING} and \ref{qloops} below.

Recall the definition of $\gamma_n$ in \eqref{gammandef} and write $M_{n} = {\rm len} ( \gamma_{n} )$ for its length.  It is proved in \cite[Theorem 1.4]{S} and \cite[Corollary 1.3]{Escape} that there exist 
\begin{equation}\label{betadef}
\beta \in (1, \frac{5}{3} ],  
\end{equation}
$c_{1} > 0$ and $c_{2} > 0$ such that for all $n \ge 1$ and $r \ge 1$,
\begin{align}
& E (M_{n} ) \asymp 2^{\beta n} \   \ \ \ \text{ and }  \label{growth} \\
&P \bigg(  \frac{M_{n}}{ E (M_{n} )} \in [r^{-1} , r] \bigg) \ge 1 - c_{1} \exp \big\{ - c_{2} r^{1/2} \big\}. \label{exp-tail}
\end{align}
We call $\beta$ the growth exponent for 3D LERW.

In this work, we will also need the following variants of the above bound on the length of the beginning part of LERW.
Let $\gamma'_n$ be the loop erasure of a simple random walk started at the origin and stopped upon first exit of $D\cap 2^{-n}\mathbb{Z}^3$ such that $B(0,4)\subset D$. It follows from \cite[Corollary 4.5]{Mas}  that there exists some universal constant $c > 0$ such that if we write $\gamma_{n}^{\infty} := \text{LE} \big( S^{(n)} [0, \infty ) \big) $ for the infinite LERW and $\lambda$ for an arbitrary path in $B^n(0,1)$ starting from $0$, then 
\begin{equation}\label{Masson}
c P ( \gamma_{n}'[0,{\rm len}(\lambda)] = \lambda ) \le P ( \gamma_{n}^{\infty} [0,{\rm len}(\lambda)]= \lambda ) \le c^{-1} P ( \gamma_{n}' [0,{\rm len}(\lambda)]= \lambda ).
\end{equation}

Set $M'_n:=T^{\gamma'_n}_{1}$ for the first time $\gamma'_n$ exits $B^{(n)} (0,1)$. To obtain similar estimates as described in \eqref{growth} and \eqref{exp-tail} for $M_{n}'$, it suffices to consider the case that $D = B (0, 4)$ thanks to \eqref{Masson}. In that case, we note that $M_{n}' \ge a$ implies that $M_{n+2} \ge a$ for any $a > 0$. (Note that, in the $D = B (0, 4)$ case, we can use the same simple random walk $S$ to construct $\gamma_{n}'$ and $\gamma_{n+2} = 2^{-2} \, \gamma_{n}'$.) Thus, it follows from \eqref{growth}, \eqref{exp-tail} and \cite[Theorem 8.12]{S} that 
\begin{equation}\label{exp-tail'}
P \bigg(  \frac{M'_{n}}{ E (M_{n} )} \in [r^{-1} , r] \bigg) \ge 1 - c_{3} \exp \big\{ - c_{4} r^{1/2} \big\}
\end{equation}
for some universal constants $0 < c_{3} , c_{4} < \infty$. Since $E (M_{n}') \le E (M_{n+2})$ in the case that $D = B (0, 4)$, combining \eqref{exp-tail'} and \eqref{growth}, we have 
\begin{equation}\label{growth'}
E (M'_{n} ) \asymp 2^{\beta n}. 
\end{equation}

  \subsubsection{Escape probability}\label{ONE}
For convenience of notation we consider the original, un-rescaled lattice $\mathbb{Z}^3$ in (and only in) this subsection. 
We will consider two independent simple random walks $S^{1}$ and $ S^{2}$ on $\mathbb{Z}^{3}$ started at the origin.
 Let $T^{i}_{N}$ denote the first time that $S^{i}$ exits from $B^{(0)}(N)$ for $i=1,2$. The following non-intersection probability is studied in \cite{Escape}. Define
  \begin{align}\label{ESCAPE}
 \text{Es} (N) = P \Big( \text{LE} \big(  S^{1} [0, T^{1}_{N} ] \big) \cap S^{2} [1, T^{2}_{N} ] = \emptyset \Big)  \ \ \text{ for }  \ N \ge 1.
   \end{align}

 Then \cite[Corollary 1.3]{Escape} proves that 
  \begin{equation}\label{ES}
  \text{Es} (N) \asymp N^{- (2 - \beta )},
  \end{equation}
  where $\beta $ is the growth exponent as described in \eqref{growth}.

 The asymptotics of the escape probability are handy tools in estimating the multi-point function of LERW. More precisely, \cite[Proposition 8.2]{S} can be  stated as follows.
\begin{prop}\label{prop:npoint1}
Let  $z_0 = 0, z_1,\cdots,z_k$ be points in $D=B^{(0)}(N)$.
There exists a universal constant $C<\infty$ such that
\begin{equation}\label{eq:npointbound}
P\Big(z_1,\cdots,z_k \in  \text{LE} \big(  S^{1} [0, T^{1}_{N} ]\big)\Big) \leq C^k\sum_{\pi\in\Pi_k}\prod_{i=1}^k G_{D}(z_{\pi_{i-1}},z_{\pi_i}){\Es}(d^{\pi_i})
\end{equation}
where $\Pi_k$ stands for the permutation group of $\{0,1,\cdots,k\}$ with $\pi_{0} =0$ for all $\pi \in \Pi_{k}$, 
\begin{align*}
&d^{\pi_{i}}=|z_{\pi_{i}}-z_{\pi_{i-1}}| \wedge |z_{\pi_{i}}-z_{\pi_{i+1}}| \wedge {\rm dist}(z_{\pi_{i}}, \partial B^{(0)}(N)) \ \ \text{ for } \ \ 1 \le i \le k-1  \\
& \text{ and }  \ \ \ d^{\pi_{k}} = |z_{\pi_{k}}-z_{\pi_{k-1}}| \wedge {\rm dist}(z_{\pi_{k}}, \partial B^{(0)}(N)). 
\end{align*}
\end{prop}
This bound can also be extended to the loop-erasure of conditioned random walks.
 Take integers $m,k,N$ with $\sqrt{3}m+k\leq N$. 
We set $A_m := [-m, m]^3\cap \mathbb{Z}^3$. We take a point $x$ lying in a face of $A_m$. We write $l$ for the infinite half line started at $x$ which lies in $A_m^c$ and is orthogonal to the face of $A_m$ containing $x$ (and choose one of  such faces arbitrarily if $x$ happens to lie on an edge of $A_m$). We write $y$ for the unique point which lies in $l$ and satisfies $|x-y|= \frac{k}{2}$. Then we let  $A_n(x):=y+[-k/4, k/4]^3\cap\mathbb{Z}^3$ 
Consider $K \subset A_m$. Suppose that $X$ is a random walk starting from $x$ conditioned that $X[0,T_{N}] \cap K=\emptyset$. 
The role of the set $A_{n} (x)$ is that $\LE(X[0,T_{N}])$ is not strongly affected by the conditioning (on the event $X[0,T_{N}] \cap K=\emptyset$) when it stays in $A_{n} (x)$. In fact, we have the following proposition, which also follows from \cite[Proposition 8.2]{S}.

\begin{prop}\label{prop:npoint2}
The bound \eqref{eq:npointbound} is also true if we replace $\gamma_n$ by $\LE(X[0,T_{N}])$ in the statement of Proposition \ref{prop:npoint1}, provided that $z_{0} = x$ and $z_{1}, \cdots , z_{k} \in A_{n} (x)$.
\end{prop}
  
  \subsubsection{The scaling limit of LERW in $d=3$}\label{SCALING} 
  Recall the definition of $\gamma_{n}$ in \eqref{gammandef}.
  We can regard $\gamma_{n}$ as an element of ${\cal H} (\overline{\mathbb{D}} )$ (see Section \ref{metric} for ${\cal H} (\overline{\mathbb{D}} )$). Kozma (\cite{Koz}) shows that $\gamma_{n}$  converges in distribution to some random compact set ${\cal K} \in {\cal H} (\overline{\mathbb{D}} )$.
  It is proved in \cite{SS} that ${ \cal K}$ is a simple curve almost surely. Moreover, in \cite{S2}, it is shown that with probability one the Hausdorff dimension of ${\cal K}$ is equal to $\beta$.

 Finally,  let 
  \begin{equation}
  \eta_{n} (t) = \gamma_{n} ( 2^{\beta n} t ) \  \ \ 0 \le t \le 2^{- \beta n} M_{n}
  \end{equation}
  be the time-rescaled version of $\gamma_{n}$, where $M_{n}$ stands for the length of $\gamma_{n}$.  
  Regarding $\gamma_{n}$ as a continuous curve with linear 
  interpolation, we may think of $\eta_{n}$ as a random element of the metric space $\big( {\cal C} (\overline{\mathbb{D}} ), \rho \big)$ (see Section \ref{metric} for this metric space). It is proved in \cite{LS} that $\eta_{n}$ 
converges in distribution to some random continuous curve $\eta : [0, t_{\eta}] \to \overline{\mathbb{D}}$ with respect to the metric $\rho$, where $0 < t_{\eta} < \infty$ stands for the time duration of the scaling limit $\eta$. This greatly strengthens Kozma's scaling limit result derived in \cite{Koz}.

 \subsubsection{``Hittability'' and Quasi-loops of LERW}\label{qloops}
 In this subsection, we introduce two path properties of the LERW, which play key roles in the analysis of the H\"older continuity of the scaling limit.
 
 We first discuss the ``hittability'' of LERW. Heuristically, this means that with high probability, the trace of a LERW is hittable by an independent simple random walk that starts from a neighbourhood of the LERW.  To be more precise, recall the definition of the LERW $\gamma_n$ in \eqref{gammandef} and 
 write $R$ for a SRW on $2^{-n} \mathbb{Z}^{3}$ which is independent of $\gamma_{n}$. The probability law of $R$ is denoted by $P^{x}_{R}$ if we assume that $R (0) = x$. Finally, we let $T^{R}_{x,r}$ stand for the first time that $R$ exits from $B (x, r)$. 
 
The following proposition is a slightly stronger version of \cite[Theorem 3.1]{SS}, adapted to our setup.
\begin{prop}\label{prop:beurling}
For any $p\geq1$ and $0<s<u$, there exist constants $\xi=\xi(p,s,u) > 0$ and $C_3=C_3(p,s,u) < \infty$ such that for all $n \ge 1$ and $\theta \in (0, 1)$,
 \begin{align}\label{beurling}
&P \Big\{   P_{R}^{x} \Big( R[0, T^{R}_{x, \theta^s } ] \cap \gamma_{n} = \emptyset \Big) \le \theta^{\xi}, \text{ for all } x \in D_{n}  \text{ with } {\rm dist} \big( x, \gamma_{n} \big) \le \theta^{u} \Big\} \ge 1 - C_3 \theta^p.
\end{align} 
Here we set $D_n=2^{-n} \mathbb{Z}^{3}\cap \mathbb{D}$.
\end{prop}
To see that Proposition \ref{prop:beurling} holds, it suffices to follow the proof of \cite[Theorem 3.1]{SS} and apply Lemmas 3.2 and 3.3 of \cite{SS} with suitable choice of parameters.
 
 \medskip
 
 ``Hittability'' estimates such as \eqref{beurling} enable us to exclude from $\gamma_{n}$ the existence of ``quasi-loops'' which we will define immediately below. 
 \begin{dfn}
We then say a path $\lambda \subset 2^{-n} \mathbb{Z}^{3}$ has an $(s_{1}, s_{2})$-{\bf quasi-loop} at $x$ for $0 < s_{1} < s_{2}$ and $x \in 2^{-n} \mathbb{Z}^{3}$ if there exist two times $k_{1} \le k_{2}$ such that 
 \begin{equation*}
 \lambda (k_{i} ) \in B ( x, s_{1})  \ \ \text{ for } \ i=1,2   \  \ \ \text{ and }  \ \ \  \lambda[k_{1}, k_{2}] \not\subset B (x, s_{2}).
 \end{equation*}
 The set of all such $x$'s is denoted by $\text{QL} ( s_{1}, s_{2}; \lambda)$.
 \end{dfn}
 
The following proposition, which states that large quasi-loops are unlikely to exist for $\gamma_n$, is a paraphrased version of \cite[Theorem 6.1]{SS}.
\begin{prop}
One can fix universal constants
 \begin{equation}\label{Ldef}
     \mbox{$L\geq 2$, $b\geq 1$, $C_4 <\infty$ }
 \end{equation}
such that for all $0<\theta <1$ and $n \ge 1$,
 \begin{equation}\label{quasi}
 P \Big( \text{QL} \big( \theta^{L}, \theta ; \gamma_{n} \big) \neq \emptyset \Big) \le C_4 \theta^{b}.
 \end{equation}
\end{prop}

Note that although \cite[Theorem 6.1]{SS} did not specify the possible value of $\alpha$ (which corresponds to $b/2$ here), it is not difficult to see that in its proof, one can fix $\alpha$ as $1/2$ in the inequality \cite[(6.1)]{SS}, ibid. by a gambler's ruin type of argument.

\subsection{Uniform spanning tree}\label{UST}
We will discuss the (wired) uniform spanning tree (WUST) in this subsection. For simplicity we will only focus on the WUST on $D_n=2^{-n} \mathbb{Z}^{3}\cap \mathbb{D}$. We regard $\partial D_{n}$ as a single vertex. With slight abuse of notation, we still write $\partial D_{n} $ for this vertex. Let  $G_n$ be the induced graph.
A spanning tree on the graph $G_{n}$ is called a {\it wired} spanning tree on $D_{n} \cup \partial D_{n}$.

A random tree obtained by choosing uniformly random among all wired spanning trees on $G_n$ is called the wired uniform spanning tree on $G_{n}$ which we denote by ${\cal U}_{n}$. 
See \cite{BLPS} or \cite[Chapter 9]{LawLim} for a more thorough introduction on the model.

To generate ${\cal U}_{n}$, we employ Wilson's algorithm (see \cite{Wil} for more details), which can be described in the following way:

 \begin{itemize}
\item[(i)] Choose an ordering of $D_{n} = \{ x_{1}, x_{2}, \cdots , x_{m} \}$ arbitrarily.

\item[(ii)] Let $R^{1}$ be the SRW started at $x_{1}$ and write $T^{1}$ for the first time that it hits $\partial D_{n}$. Let ${\cal U}^{1}$ be the tree formed of the edges of
 the loop-erasure $\text{LE} (R^{1} [0, T^{1} ] )$.

\item[(iii)] Take another SRW $R^{j+1}$ started from $x_{j+1}$ for $j \ge 1$. Let $T^{j+1}$  be the first time that $R^{j+1}$ hits ${\cal U}^{j}$. Let ${\cal U}^{j+1}$ be the tree formed of the union of the edges in $ \text{LE} (R^{j+1} ) $ and ${\cal U}^{j}$.

\item[(iv)] Repeat Step (iii) and stop when all vertices in $D_{n}$ are contained in the tree.
\end{itemize}
In \cite{Wil}, it is shown that the distribution of the final output tree ${\cal U}_{n}$ is the same as the WUST on $G_n$ for any ordering of $D_{n}$. 

Moreover, the LERW $\gamma_{n}$ has the same distribution as $\gamma_{0, \partial}$ where $\gamma_{0, \partial}$ stands for the unique simple path in ${\cal U}_{n}$ connecting the origin and $\partial D_{n} $.
Finally, we write $d_{{\cal U}_{n}} ( \cdot , \cdot )$ for the graph distance on ${\cal U}_{n}$.

\section{The upper bound of the modulus of continuity}\label{sec:3}
The main goal of this section is to prove the first claim of Theorem \ref{main}, as paraphrased in Proposition \ref{0626-a-1}, where we show the $h$-H\"older continuity of the scaling limit for all $h < {1}/{\beta}$. The crucial step is to derive modulus-of-continuity results on the rescaled LERW, which we summarize in Proposition \ref{HC-upper}. To obtain this proposition, in Section \ref{sec:3.1}, we derive some regularity properties for the LERW path and then, as discussed below Remark \ref{mainrmk}, in Section \ref{sec:3.2}, we apply an epsilon-net argument with the help of results from Section \ref{sec:3.1} to control the length of the whole LERW path with some tail probability estimates which which were originally only available for the beginning part of the LERW.

A summary of notation employed in this section is offered in Table \ref{symbols-3} at the end of this section.

\subsection{Regularity of LERW paths}\label{sec:3.1}

In this subsection, we will prove Proposition \ref{lem:Festimate}, which summarizes various estimates on regularity properties of a LERW path. In particular, we show that with high probability, a LERW path 
\begin{itemize}
    \item{\bf a)} can be covered by not too many small balls;
    \item{\bf b)} can be hit by a simple random walk starting nearby;
    \item{\bf c)} does not contain big ``quasi-loops''.
\end{itemize}
See \eqref{holder-a-2} for precise descriptions of these properties.

We now consider 
\begin{equation}
    \mbox{$\epsilon \in (0, 1/10 )$ and  $\delta \in (0, 1/10)$ such that $\delta^{-\epsilon/2} > 100$}.
\end{equation} We pick the value of $\epsilon$ later according to the value of $h$ in the proof of Proposition \ref{0626-a-1} and let $\delta$ tend to 0 gradually in the same proof.

Recall the definition of $\gamma_{n}$ in \eqref{gammandef}. We define a sequence of random times $\tau_{l}$ created by $\gamma_{n}$ as follows. Let $\tau_{0} = 0$. We define
\begin{equation}\label{holder-2-2-1}
\tau_{l} := \inf \Big\{ j \ge \tau_{l-1} \ \Big| \ | \gamma_{n} (j) - \gamma_{n} (\tau_{l-1} ) | \ge  \delta^{\frac{1}{\beta} - \epsilon_{0} } \Big\},  \ \ \ \text{ for } l \ge 1,
\end{equation}
where 
\begin{equation}\label{holder-0}
\epsilon_{0}:=\epsilon/100.
\end{equation}

We write 
\begin{equation}\label{0726-a-1}
N: = \max \, \{ l \ge 1 \ | \ \tau_{l} \le {\rm len} (\gamma_{n} ) \}.
\end{equation}

Writing $x_{l} = \gamma_{n} (\tau_{l} )$, we define three events $F_{(i)}$ ($i=1,2,3$) as follows (note that they correspond to properties {\bf a)} through {\bf c)} discussed at the beginning of this subsection): 
\begin{align}\label{holder-a-2}
&F_{(1)} = \big\{ N \le \delta^{-3}  \big\},  \notag \\
&F_{(2)} = \Big\{ \text{For all $x \in D_{n}$ with $\text{dist} (x , \gamma_{n} ) \le r $, } P^{x}_{R} \Big( R \big[ 0, T^{R}_{x, \sqrt{r}} \big] \cap \gamma_{n} = \emptyset \Big) \le \delta^{5} \Big\},  \\
&F_{(3)} = \Big\{ \text{For all $l= 0, 1, \cdots , N$, }  \big( \gamma_{n} [0, \tau_{l-1} ] \cup \gamma_{n} [\tau_{l +1} , {\rm len}  (\gamma_{n} ) ]  \big) \cap B \big( x_{l}, r^{\frac{1}{3}} \big) = \emptyset \Big\}\notag 
\end{align}
(with convention that $\gamma_{n} [0, \tau_{-1} ] = \emptyset$ in the event $F_{(3)}$), where $L$ is the constant in \eqref{Ldef} and 
\begin{equation}\label{rdef}
    r:=\delta^{\max \{ 5/\xi(10,1/2,1),12L \}}
\end{equation}
where $\xi(\cdot,\cdot,\cdot)$ is the parameter defined in Proposition \ref{prop:beurling}. The motivation of choosing such an $r$ will be clear later as we estimate the asymptotic probability of the events defined above.

Let 
\begin{equation}
    F:=F_{(1)}\cap F_{(2)} \cap F_{(3)}.
\end{equation}
We now claim that $F$ happens with very high probability. 
\begin{prop}\label{lem:Festimate}
There exists $C_5 <\infty$ and $c_6>0$ such that for all $n\geq 0$,
\begin{equation}\label{holder-2-3}
P ( F ) \ge 1 - C_5 \delta^{c_6}.
\end{equation}
\end{prop}
\begin{proof}
The claim follows from Lemmas \ref{lem:F1} through \ref{lem:F3}.
\end{proof}
We now estimate the probability of $F_{(i)}$'s separately.
\begin{lem}\label{lem:F1}
There exist $c >0$ and $C < \infty$ such that
\begin{equation}\label{N}
P(F_{(1)})\geq 1- C \exp \big\{ - c \delta^{-c} \big\}.
\end{equation}
\end{lem}
This lemma is in essence a coarse version of H\"older continuity result we seek for in this work, and although in principle we can give much more precise bounds than this one here, this bound already suffices for the estimates to follow.
\begin{proof}
 Consider $S$ the SRW started at the origin on $2^{-n} \mathbb{Z}^{3}$ whose loop-erasure is $\gamma_n$.  In order to give a bound on $N$, we also define a sequence of stopping times $\tau_{l}'$ for $S$, similar to $\tau_l$ for $\gamma_n$: set $\tau_{0}' = 0$ and let
\begin{equation*}
\tau_{l}' = \inf \Big\{ j \ge \tau_{l-1}' \ \Big| \ | S (j) - S (\tau_{l-1}' ) | \ge 3^{-1} \cdot  \delta^{\frac{1}{\beta} - \epsilon_{0}} \Big\},  \ \ \ \text{ for } l \ge 1.
\end{equation*}
Set $N'$ as the largest integer satisfying that $\tau_{N'} \le T_1$. Observe that there is an injective map $\Sigma:\{0,\ldots,N\}\to\{0,\ldots,N'\}$ such that $\gamma_{\tau_j}\in B \Big( S \big( \tau'_{\Sigma(j)} \big), 3^{-1} \delta^{\frac{1}{\beta} - \epsilon_{0}} \Big) $, hence $N \le N'$.
Therefore, the proof is completed once we obtain the inequality as follows.
\begin{equation}\label{holder-4}
 P ( N' \ge \delta^{-3} )  \le C \exp \big\{ - c \delta^{-c} \big\}.
\end{equation}

To prove \eqref{holder-4}, let  $\upsilon = 3^{-1} \cdot  \delta^{\frac{1}{\beta} - \epsilon_{0}}$. It follows from \cite[Proposition 2.4.5]{LawLim} that 
\begin{equation*}
P \Big( T_{1}  \ge \delta^{-\frac{1}{2}} \, 2^{2n}  \Big) \le C \exp \big\{ - c \delta^{-\frac{1}{2}} \big\}
\end{equation*}
and
\begin{equation*} P \Big( \tau_{l}' - \tau_{l-1}' \le \delta^{\frac{1}{2}} \, \upsilon^{2} \, 2^{2n} \Big) \le C \exp \big\{ - c \delta^{-\frac{1}{2}} \big\}  \  \text{ for each } l \ge 1.
\end{equation*}
Setting ${\cal H}_{l} = \big\{ \tau_{l}' - \tau_{l-1}' \ge \delta^{\frac{1}{2}} \, \upsilon^{2} \, 2^{2n}  \big\}$ for $l \ge 1$, we can conclude that 
\begin{align*}
&\quad \; P ( N' \ge \delta^{-3} )  \le P \Big( N' \ge \delta^{-3}, \ {\cal H}_{l}  \text{ holds for all } l =1, \cdots ,  \lfloor \delta^{-3} \rfloor  \Big)  +  C \delta^{-3}  \exp \big\{ - c \delta^{-\frac{1}{2}} \big\} \\
&\le P \Big(  T_{1}  \ge \tau_{\lfloor \delta^{-3} \rfloor }', \  {\cal H}_{l}  \text{ holds for all } l =1, \cdots ,  \lfloor \delta^{-3} \rfloor  \Big)  +  C \delta^{-3}  \exp \big\{ - c \delta^{-\frac{1}{2}} \big\} \\
&\le P \big( T_{1}  \ge c \delta^{-\frac{1}{2}} \, 2^{2n} \big) + C \delta^{-3}  \exp \big\{ - c \delta^{-\frac{1}{2}} \big\} \le C \exp \big\{ - c \delta^{-c} \big\},
\end{align*}
where we used $\beta > 1$ in the third inequality to show that 
$$\lfloor \delta^{-3} \rfloor \cdot \delta^{\frac{1}{2}} \, \upsilon^{2} \, 2^{2n}  \ge c \, \delta^{-3 + \frac{1}{2} + \frac{2}{\beta} - 2 \epsilon_{0}} \, 2^{2n} \ge c \, \delta^{-\frac{1}{2}} \, 2^{2n}.$$
(Here $\lfloor a \rfloor$ stands for the greatest integer less than or equal to $a$.)
This finishes the proof.
\end{proof}

Recall the definition of $C_3$ in Proposition \ref{prop:beurling}.
\begin{lem}\label{lem:F2}
$$
P(F_{(2)})\geq 1-C_3r^{10}.
$$
\end{lem}
Note that the choice of $10$ is arbitrary.
\begin{proof}
The claim follows directly from \eqref{beurling} by letting 
$r=\theta$, $(p,s,u)=(10,1/2,1)$ and noting that $r^{\xi(10,1/2,1)}\leq\delta^{5}$ by definition of $r$ in \eqref{rdef}.
\end{proof}

\begin{lem}\label{lem:F3}
There exist $C_8 <\infty$ and $c_9 >0$ such that
 $$ P(F_{(3)})\ge 1 - C_8 \delta^{c_9}.$$
\end{lem}
\begin{proof}
Since $| x_{l} - x_{l-1}| \ge  \delta^{\frac{1}{\beta} - \epsilon_{0}}$ for each $l \ge 1$, the event $F_{(3)}^{c}$ implies that $$\text{QL} \big(  r^{1/3},  \delta^{\frac{1}{\beta} - \epsilon_{0}}; \gamma_{n} \big) \neq \emptyset,$$ where $\text{QL} ( \cdot, \cdot ; \cdot )$ is defined as in Section \ref{qloops}. 
The claim then holds by applying  \eqref{quasi} with
$\theta = r^{\frac{1}{3L}}$. By the definition of $r$ in \eqref{rdef}, it follows that $r \le \delta^{12 L}$ and 
$$
\theta \le \delta^{4} < \delta^{\frac{1}{\beta} - \epsilon_{0}},
$$
where we used $\beta > 1$ in the last inequality. 
Thus, $\text{QL} \big(  r^{1/3},  \delta^{\frac{1}{\beta} - \epsilon_{0}}; \gamma_{n} \big) \subseteq \text{QL} \big(  \theta^{L}, \theta; \gamma_{n} \big).$
\end{proof}

\subsection{The epsilon-net argument}\label{sec:3.2}
In this subsection, we carry out the epsilon-net argument announced at the beginning of this section. To be more specific, we first use Wilson's algorithm as described in Section \ref{UST} to generate the uniform spanning tree ${\cal U}_{n}$ on the graph $G_{n}$ that contains $\gamma_n$, then consider the LERW paths on this tree started from points in an epsilon-net and finally show that with high probability the beginning parts of these paths are sufficiently regular and these beginning parts can already cover $\gamma_n$.

We start with the setup of epsilon-net.
Recall the choice of $r$ in \eqref{rdef}. We consider a {\bf \textit{net}} $G = \{ y^{j} \}_{j=1}^{M} \subset D_{n}$ in the sense that 
\begin{equation}\label{0726-b-1}
\mbox{for all $x \in D_{n}$ there exists $j = 1,2, \cdots , M$ such that $x \in B (y^{j}, r )$ and $M\asymp r^{-3}$.}
\end{equation}
For each $l=0,1, \cdots , N$,  let $y_{l} := y^{j_{l}} \in G$ be one of the nearest points from $x_{l} := \gamma_{n} (\tau_{l} )$ among $G$ so that $|y^{j_l} - x_{l} | \le r$, where $\tau_{l}$ and $N$ are defined in \eqref{holder-2-2-1} and \eqref{0726-a-1} resp. and $|\cdot|$ stands for the Euclidean norm. Recall the definition of Wilson's algorithm in Section \ref{UST}. As the order of vertices from which LERW's start are arbitrary, we can perform Wilson's algorithm  with a specific order of vertices as follows:
\begin{align}
&\text{(i) \ \  Set ${\cal U}^{0}= \gamma_{n}$;} \notag \\
&\text{(ii) \ For each $l=0,1, \cdots , N$, consider a SRW $R^{l}$ started at $y_{l}$, let $t^{l}\geq 0$ be the first time}\notag\\
&\ \ \ \ \ \text{ that $R^{l}$ hits ${\cal U}^{l} \cup \partial D_{n}$ and set } {\cal U}^{l+1} = {\cal U}^{l} \cup \text{LE} \big( R^{l} [0, t^{l} ] \big);\notag \\ 
&\text{(iii) Perform Wilson's algorithm for all points in $D_{n} \setminus {\cal U}^{N+1}$ in an arbitrary order,} \label{0726-c-1}
\end{align}
and note that the resulting tree $\cal U$ has the distribution of a WUST on $D_n\cup \partial D_n$.
See Figure \ref{fig:wil-alg} below for an illustration.
\begin{figure}[h]
\begin{center}
\includegraphics[scale=0.7]{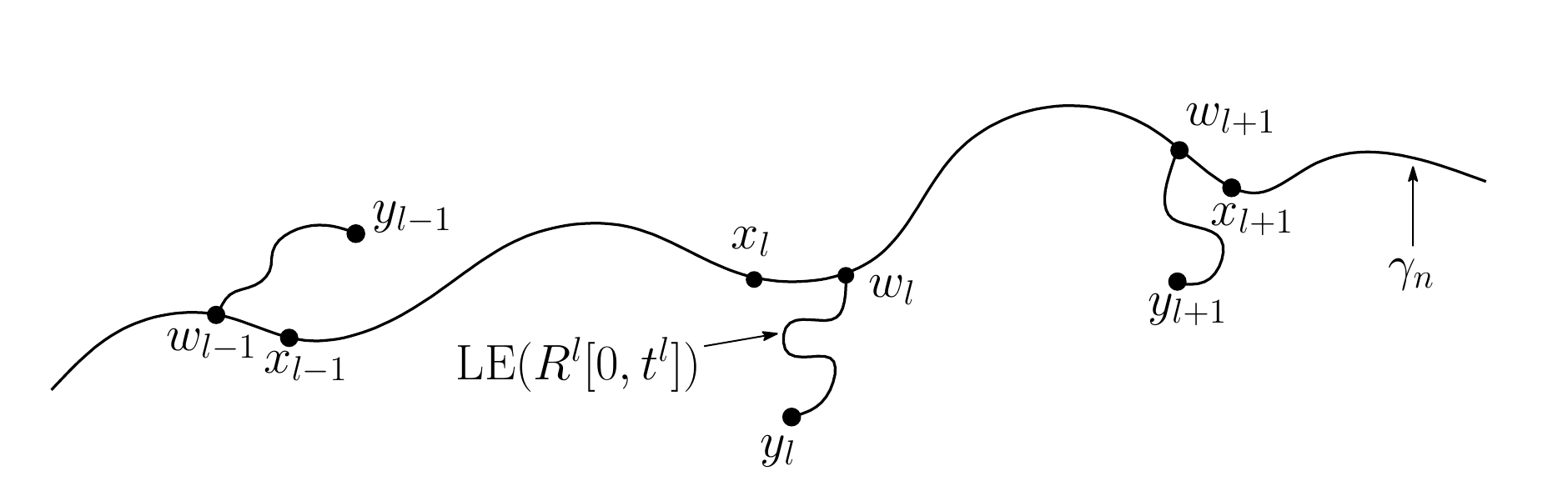}
\caption{Illustration for the step (ii) of Wilson's algorithm}\label{fig:wil-alg}
\end{center}
\end{figure}

In the rest of this subsection, it is convenient to think of $\gamma_{n}$ as a deterministic set and condition $\gamma_{n}$ on the event $F= F_{(1)} \cap F_{(2)} \cap F_{(3)}$. 
We denote the endpoint of $R^{l}$ by $w_{l} := R^{l} (t^{l})$. We define the event $H$ by
\begin{equation}\label{0726-h}
\begin{split}
 H = \Big\{  R^{l} [0, t^{l} ] \subset B (y_{l}, \sqrt{r} )  &\text{ and } w_{l} \in \gamma_{n} [\tau_{l-1}, \tau_{l+1} ]\\
 & \quad\text{ for all } l= 0,1, \cdots , N  \mbox{ s.t. } \text{dist} (y_l,\partial \mathbb{D})\geq 4\sqrt{r} \Big\},
 \end{split}
\end{equation}
with the convention that $\tau_{-1}=0$ and $\tau_{N+1}={\rm len}(\gamma_n)$. The event $H$ controls the diameter of $R^{l}$ and the location of its endpoint $w_{l}$ for every $l$ for which $y_{l}$ is not too close to $\partial \mathbb{D}$. We will now show that on the event $F$, the conditional probability of the event $H$ is close to $1$.

\begin{lem}\label{lem:HF}
There exists $C < \infty$ such that
\begin{equation}\label{holder-3-1}
P(H|F) \geq  1 - C\delta^{2}.
\end{equation}
\end{lem}

\begin{proof}
The event $F_{(2)}$ ensures that for all $l$ such that $\text{dist} (y_l,\partial \mathbb{D})\geq 4\sqrt{r}$,
\begin{equation*}
P^{y_{l}}_{R^{l}} \Big(  R^{l} [0, t^{l} ] \not\subset B (y_{l}, \sqrt{r} ) \ \Big| \ F \Big)  \le P^{y_{l}}_{R^{l}} \Big(  R^{l} \big[ 0, T^{R^{l}}_{y_{l}, \sqrt{r}} \big] \cap \gamma_{n} = \emptyset \ \Big| \ F \Big) \le   \delta^{5}.
\end{equation*}
Since $N \le \delta^{-3}$ by the event $F_{(1)}$, we see that 
$$
P\Big(\mbox{$R^{l} [0, t^{l} ] \subset B (y_{l}, \sqrt{r} ) $ for all $l= 0,1, \cdots , N$ s.t.\ }\text{dist} (y_l,\partial \mathbb{D})\geq 4\sqrt{r} \ \Big|  \ F\Big)\geq 1 - C \delta^{2 }.
$$
We now note that on the event 
$$
H':=\big\{\mbox{$R^{l} [0, t^{l} ] \subset B (y_{l}, \sqrt{r} ) $ for all $l= 0,1, \cdots , N$ s.t.\ } \text{dist} (y_l,\partial \mathbb{D})\geq 4\sqrt{r}\big\},
$$
by the definition of the event $F_{(3)}$, $| x_{l} - x_{l-1} | \ge \delta^{\frac{1}{\beta} - \epsilon_{0} }$ and $|x_{l} - y_{l} | \le r$, 
it holds necessarily that 
$w_{l} = R^{l} (t^{l} ) \in \gamma_{n} [\tau_{l-1}, \tau_{l+1} ] \subset \gamma_n$, in other words, $R^{l}$ must not hit branches $\text{LE} \big( R^{i} [0, t^{i} ] \big) $ ($i=0,1, \cdots , l-1$) generated earlier in \eqref{0726-c-1} (ii). 

Now assume that on the event $H'\cap F$, $w_{l} \notin \gamma_{n} [\tau_{l-1}, \tau_{l+1} ]$. This implies that $w_{l} \in  \gamma_{n} [0, \tau_{l-1} ] \cup \gamma_{n} [\tau_{l +1} , {\rm len}  (\gamma_{n} ) ]  $. However, since $| w_{l} - y_{l} | \le \sqrt{r}$ and $|x_{l} - y_{l} | \le r$, our choice of $r$ (see \eqref{rdef}) ensures that $w_{l} \in B \big( x_{l}, r^{\frac{1}{3}} \big)$. This contradicts the event $F_{(3)}$. 
This finishes the proof of 
\eqref{holder-3-1}.
\end{proof}

We also need to control the length of  $\text{LE} (R^{l} [0, t^{l} ] )$. Write $L_{l} = {\rm len} \Big( \text{LE} (R^{l} [0, t^{l} ] ) \Big) $ and 
define the event $I$ as
\begin{equation}\label{0726-i}
I: = \Big\{ L_{l} \le r^{\frac{1}{3}} 2^{\beta n} \text{ for all } l = 0, 1, \cdots , N  \mbox{ s.t. } \text{dist} (y_l,\partial \mathbb{D})\geq 4\sqrt{r} \Big\}.
\end{equation}
Note that $R^{l} [0, t^{l} ]$ is with high probability contained in $B (y_{l}, \sqrt{r} )$ on  the event $F\cap H$. Thus, in this case the length $L_{l}$ must be typically smaller than $r^{\frac{\beta}{2}} 2^{\beta n}$. Since $\beta \in (1, 5/3]$, the event $I$ must occur with high probability. More precisely, we have the following lemma.
\begin{lem}\label{lem:FHIc}
There exist $c > 0$ and $C < \infty$ such that
\begin{equation}\label{eq:FHIc}
  P ( H \cap I^{c} | F ) \le C r^{-3} \exp \big\{ -c r^{-\frac{1}{12}} \big\}.
\end{equation}
\end{lem}
\begin{proof}
Suppose that the event $H \cap I^{c}$ occurs. Then we have $L_{l} \ge   r^{\frac{1}{3}} 2^{\beta n}$ for some $l$. The event $H$ guarantees that $  R^{l} [0, t^{l} ]$ is contained in $B (y_{l}, \sqrt{r} )$. With this in mind, we write $\gamma_{x, \partial }$ for the unique simple path in ${\cal U}_{n}$ starting from $x$ to $\partial D_{n}$. Then, setting \begin{equation}\label{Udef}
U_{x}:=T^{\gamma_{x,\partial}}_{x,\sqrt{r}}
\end{equation}
as the first time that $\gamma_{x, \partial }$ exits from $B (x, \sqrt{r} )$, the event $H \cap I^{c}$ ensures that $U_{y_{l}} \ge r^{\frac{1}{3}} 2^{\beta n}$ for some $l$. But, for each $x \in D_{n}$, the law of $\gamma_{x, \partial }$ coincides with that of $\text{LE} \big( R^{x} [0, t_{x} ] \big)$ where $R^{x}$ denotes a SRW started at $x$ and $t_{x}$ stands for the first time that $R^{x}$ hits $\partial D_{n}$. Combining this fact with  \eqref{exp-tail'}, we have
\begin{equation}
\label{holder-c-1}
\begin{split}
&  P ( H \cap I^{c} | F )  \\
\le \;&P \Big( \text{There exists } j \in \{ 1, 2, \cdots , M \} \text{ such that }  \text{dist} (y^{j},\partial \mathbb{D})\geq 4\sqrt{r} \mbox{ and } U_{y^{j}} \ge r^{\frac{1}{3}} 2^{\beta n} \Big)  \\
\le\;& \sum_{j=1,\; \text{dist} (y^j,\partial \mathbb{D})\geq 4\sqrt{r} }^{M} P \big( U_{y^{j}} \ge r^{\frac{1}{3}} 2^{\beta n} \big) \le C r^{-3} \exp \big\{ -c r^{-\frac{1}{12}} \big\},
\end{split}
\end{equation}
where we recall that the epsilon net consists of   
$M \asymp r^{-3}$ points, the exponent $1/12$ comes from the fact that $E[U_{y^j}] \asymp r^{\beta/2}2^{\beta n}$ if $\text{dist} (y^j,\partial \mathbb{D})\geq 4\sqrt{r}$, and the observation that $(\beta/2-1/3)/2>1/12$ in applying \eqref{exp-tail'}. We thus finish the proof.
\end{proof}

\subsection{Lower bound on the modulus of continuity}
With all preparations complete, we can now give the lower bound of the modulus of continuity.

\begin{prop}\label{HC-upper}

Let
\begin{equation}\label{0630-w}
K_n := \bigg\{ \exists x' = \gamma_{n} (s) \text{ and }   y' = \gamma_{n} (t)  \text{ s.t.\ } 0 < t-s \le 2  \delta 2^{\beta n} \text{ and } |x'-y'| \ge \frac{\delta^{\frac{1}{\beta} - \epsilon }}{2} \bigg\}.
\end{equation}
Then, there exist universal constants $c, C \in (0, \infty )$ such that for all $\delta \in (0,1/10)$, $\epsilon \in (0, \frac{1}{10})$ with $\delta^{-\frac{\epsilon}{2}} > 100$ 
\begin{equation}\label{J}
 P (K_n^c)  \ge 1 - C \delta^{c \epsilon}.
 \end{equation}
\end{prop}

\begin{proof}

We want to show that the event $K_n$ is unlikely to happen. For that purpose, suppose that the event $F \cap H \cap I \cap K_n$ occurs. 

Recall the notion of $x'$, $y'$, $s$ and $t$ in the definition of the event $K_n$ in \eqref{0630-w}. Find $i < j$ satisfying that 
\begin{equation}\label{ij}
\text{$\tau_{i-1} \le s < \tau_{i}$ and $\tau_{j-1 } \le t < \tau_{j}$}.
\end{equation}
 Since the diameter of $\gamma_{n} [\tau_{l-1}, \tau_{l} ]$ is bounded above by $\delta^{\frac{1}{\beta} - \epsilon_{0}}$ for each $l$ and $|x'- y'| \ge \frac{\delta^{\frac{1}{\beta} - \epsilon}}{2}$, it holds that $j -i \ge \delta^{- \epsilon /2 } > 100$, where we recall that $\epsilon_{0} = 10^{-2} \epsilon$ as defined in \eqref{holder-0}. The graph distance in ${\cal U}_{n}$ is denoted by $d_{{\cal U}_{n}} ( \cdot , \cdot )$. 

We now distinguish between two cases, depending on whether $y_{i+1}$ is close to the boundary of $\mathbb{D}$.

Suppose first that $\text{dist} ( y_{i+1},\partial\mathbb{D})>\delta':= \delta^{\frac{1}{\beta} - \epsilon/10} \ge 4 \sqrt{r}$. Recall that $$L_{l} = {\rm len} \Big( \text{LE} (R^{l} [0, t^{l} ] ) \Big) $$ as defined above \eqref{0726-i}. Since the event $H \cap I \cap K_{n}$ occurs, it holds that  
\begin{equation*}
\begin{split}
d_{{\cal U}_{n}} (y_{i + 1}, y' ) &\;= d_{{\cal U}_{n}} (y_{i + 1}, w_{i+1}) + d_{{\cal U}_{n}} (w_{i + 1}, y' ) \\&\;= L_{i+1} + d_{{\cal U}_{n}} (w_{i + 1}, y' )\leq (2 \delta + r^{\frac{1}{3}} ) 2^{\beta n} \le  3 \delta 2^{\beta n},
\end{split}
\end{equation*}
where in the second to last inequality we observe that by definition of $w_{i+1}$, $s<\tau_i\leq w_{i+1}$, and hence by the definition of $K_n$, $d_{{\cal U}_{n}} (w_{i + 1}, y' )\leq d_{{\cal U}_{n}} (x', y' )=t-s$. Moreover,
the event $H$ ensures that $|y_{i+1} - y'| \ge |x'- y'| - |x'- x_{i}| - |x_{i} - x_{i+1}| - |x_{i+1} - w_{i+1} | - |w_{i+1}- y_{i+1}| \ge  \frac{\delta^{\frac{1}{\beta} - \epsilon }}{2} - 3 \delta^{\frac{1}{\beta} - \epsilon_{0}} - \sqrt{r} \ge   \delta'/4$. Consequently, we see that the first time that $\gamma_{y_{i+1}, \partial }$ exits from $B \big( y_{i+1}, \delta'/4 \big)$ (we denote this first exit time by $U'_{y_{i+1}}$)  is smaller than $3 \delta 2^{\beta n}$. Here we recall that for each $x \in D_{n}$, $\gamma_{x, \partial }$ stands for the unique simple path in ${\cal U}_{n}$ starting from $x$ to $\partial D_{n}$. The law of $\gamma_{x, \partial }$ coincides with that of $\text{LE} \big( R^{x} [0, t_{x} ] \big)$ where $R^{x}$ denotes the SRW started at $x$ and $t_{x}$ stands for the first time that $R^{x}$ hits $\partial D_{n}$. In this case we can  use \eqref{exp-tail'} again as in \eqref{holder-c-1} to show that for any $y^l$ such that $\text{dist} (y^l,\partial \mathbb{D})>\delta'$ there exist $c,C \in (0, \infty)$ such that
\begin{equation}\label{eq:deepinside}
    P \big( U'_{y^{l}} \le 3 \delta 2^{\beta n} \big)\leq C \exp \Big\{ - c \delta^{- c \epsilon} \Big\}.
\end{equation}

We will next deal with the the remaining case that $\text{dist} ( y_{i+1},\partial\mathbb{D} )  \le \delta':= \delta^{\frac{1}{\beta} - \epsilon/10}$. Note that 
\begin{equation*}
|x_{i-1} - y_{i+1} | \le |x_{i-1} - x_{i} | + |x_{i} - x_{i+1} | + |x_{i+1} - y_{i+1} | \le 2 \delta^{\frac{1}{\beta} - \frac{\epsilon}{100}} + r \le 3  \delta^{\frac{1}{\beta} - \frac{\epsilon}{100}},
\end{equation*}
and thus $\text{dist} (x_{i-1}, \partial \mathbb{D} ) \le \delta^{\frac{1}{\beta} - \epsilon/10} + 3  \delta^{\frac{1}{\beta} - \frac{\epsilon}{100}} \le 2 \delta^{\frac{1}{\beta} - \epsilon/10}$. Since $\tau_{i-1} \le s \le t $, $x_{i-1} = \gamma_{n} (\tau_{i-1} )$ and $| \gamma_{n} (s) - \gamma_{n} (t) | \ge \frac{\delta^{\frac{1}{\beta} - \epsilon }}{2} $ in view of the event $K_{n}$, this implies that the diameter of $$S^{(n)} \Big[ T_{1 - 2 \delta^{\frac{1}{\beta} - \epsilon/10}} \, ,  T_{1} \Big] $$  is bounded below by $2^{-1} \cdot \delta^{\frac{1}{\beta} - \epsilon } $, where we recall that $T_{r}$ stands for the first time that SRW $S^{(n)}$ exits from $B^{(n)} (r)$. This probability is  bounded above by $ C \delta^{c \epsilon} $ for some universal constants $0 < c , C < \infty$. To see this, set $\varpi = 2 \delta^{\frac{1}{\beta} - \epsilon/10}$, 
\begin{align*}
\varsigma_{0} = T_{1 - \varpi } \ \  \text{ and } \ \  \varsigma_{i} = \inf \Big\{ j \ge \varsigma_{0}  \ \Big| \ \big| S^{(n)} (j) - S^{(n)} ( \varsigma_{0} ) \big| \ge  2^{i} \varpi  \Big\} \ \ \text{ for } i \ge 1.
\end{align*}
Let ${\cal N} $ be the largest integer such that $2^{{\cal N}} \varpi \le 2^{-2} \cdot \delta^{\frac{1}{\beta} - \epsilon }$. Note that $2^{{\cal N}} \asymp \delta^{-\frac{9}{10} \epsilon} $. The strong Markov property ensures that there exists some universal constant $c_{1} > 0$ such that for $1 \le i \le {\cal N}$ and each path $\lambda$ with $ P \big( S^{(n)} [ 0, \varsigma_{i-1}]  = \lambda \big) > 0$
$$P \Big( S^{(n)} [ \varsigma_{i-1}, \varsigma_{i}] \cap B^{(n)} (1)^{c} \neq \emptyset \ \Big| \  S^{(n)} [ 0, \varsigma_{i-1}]  = \lambda \Big) \ge c_{1}.$$
Iterating this and writing $1-c_{1} = 2^{-c_{2}}$ with $c_{2} > 0$, we can conclude that the probability that the diameter of $S^{(n)} [ \varsigma_{0}, T_{1} ]$ is bounded below by $2^{-1} \cdot \delta^{\frac{1}{\beta} - \epsilon } $ is smaller than $(1-c_{1})^{{\cal N}} = 2^{- c_{2} {\cal N}} \le  C \delta^{ \frac{9}{10} c_{2} \epsilon } $, as desired.

Combining these two cases, we have 
\begin{align*}
& P ( K_{n} ) \le P \big ( F \cap H \cap I \cap K_{n} \big)  + C \delta^{c}  \\
\le\;& P \Big( F \cap H \cap I \cap K_{n} \cap \big\{ \text{dist} ( y_{i+1},\partial\mathbb{D})>\delta' \big\} \Big) \\
&\quad+ P \Big( F \cap H \cap I \cap K_{n} \cap \big\{ \text{dist} ( y_{i+1},\partial\mathbb{D}) \le \delta' \big\} \Big) + C \delta^{c} \\
\le\;& P \Big( \text{There exists $l = 1, 2, \cdots , M$ such that $\text{dist} (y^l,\partial \mathbb{D})>\delta'$ and  $ U'_{y^{l}} \le 3 \delta 2^{\beta n}  $} \Big) \\
\phantom{+}\;&\quad + P \Big( \text{The diameter of $S^{(n)} \Big[ T_{1 - 2 \delta^{\frac{1}{\beta} - \epsilon/10}} \, ,  T_{1} \Big] $ is bigger than $\frac{\delta^{\frac{1}{\beta} - \epsilon }}{2}$} \Big) + C \delta^{c} \\
\le\;&  \sum_{l=1}^{M} C \exp \Big\{ - c \delta^{- c \epsilon} \Big\} + C \delta^{c \epsilon}  +  C \delta^{c} \le C M \exp \Big\{ - c \delta^{- c \epsilon} \Big\} + C \delta^{c \epsilon} \\
 \le\;& C r^{-3} \exp \Big\{ - c \delta^{- c \epsilon} \Big\} + C \delta^{c \epsilon} \le C \delta^{ c \epsilon},
\end{align*}
where we used 
\begin{itemize}
\item Proposition \ref{lem:Festimate}, Lemma \ref{lem:HF} and Lemma \ref{lem:FHIc} in the first inequality to estimate $P \big( (F \cap H \cap I )^{c} \big)$;

\item $M \asymp r^{-3}$ in the sixth inequality;

\item the fact that  $r^{-3}$ is polynomial in $\delta^{-1}$ in the last inequality.

\end{itemize}
So, we finish the proof.
\end{proof}

Finally in the next proposition we show that $\eta$ (recall its definition in Section \ref{SCALING}), the scaling limit of $\{ \eta_{n} \}_{n \ge 1}$, is $h$-H\"{o}lder continuous for any $h < 1/\beta$, which constitutes the first part of Theorem \ref{main}.
\begin{prop}\label{0626-a-1}
The random curve $\eta$ is almost surely $h$-H\"{o}lder continuous for any $h < 1/\beta$. Namely, with probability one we have for any $h < 1/\beta$
\begin{equation}\label{0626-d-1}
\sup_{0 \le s < t \le t_{\eta} } \frac{|\eta (s) - \eta (t) |}{|s-t|^{h}} < \infty.
\end{equation}
Here $t_{\eta}$ stands for the time duration of $\eta$. 
\end{prop}

\begin{proof}
Skorokhod's representation theorem asserts that we can couple $\{ \eta_{n} \}_{n \ge 1 }$ and $\eta$ in the same probability space such that $\rho ( \eta_{n} , \eta ) \to 0$ almost surely. Here we recall that the metric $\rho$ is defined as in Section  \ref{metric}.
From now on, we assume this coupling. We note that the inequality \eqref{exp-tail} guarantees that $t_{\eta} \in (0, \infty)$ almost surely.

 Without loss of generality, pick
 \begin{equation}\label{hdef}
 h\in\left(\frac{1}{\beta}-\frac{1}{10}, \frac{1}{\beta}\right),
 \end{equation}
and set 
\begin{equation}
    \epsilon =\frac{1}{\beta} - h.
\end{equation} 
Define the event $\widetilde{A}(\delta)$ (and note that it is measurable in the Borel $\sigma$-algebra generated by the metric $\rho$) by 
\begin{equation*}
\widetilde{A}(\delta) = \Big\{ \text{There exist } s, t \text{ with } 0 \le s < t \le t_{\eta} \text{ such that } t-s \le \delta \text{ and } |\eta (s) - \eta (t) | \ge \delta^{h} \Big\}.
\end{equation*}

\begin{table}[!h]
  \centering
  \begin{tabular}{l|cr}
    \hline \hline
    Symbol  & Definition and first appearance    \\
    \hline \hline
    $\gamma_n$ &  Loop-erased random walk on $2^{-n}\mathbb{Z}^3$ in $\mathbb{D}$ defined in \eqref{gammandef} \\ \hline
    ${\cal U}_{n}$ &  Wired uniform spanning tree on $G_{n} = D_{n} \cup \partial D_{n}$ from Section \ref{UST} \\ \hline
    $d_{{\cal U}_{n}} (\cdot, \cdot ) $ & Graph distance on  ${\cal U}_{n}$ defined in Section \ref{UST}\\ \hline 
    $\{ \tau_{l} \}_{l=1}^{N}$ & Sequence of random times as defined in \eqref{holder-2-2-1} and \eqref{0726-a-1} \\ \hline
    $x_{l} $ & $ \gamma_{n} ( \tau_{l} ) $ defined below \eqref{0726-a-1} \\ \hline 
    $F_{(i)} $ & Events defined  in \eqref{holder-a-2} \\ \hline
    $G = \{ y^{j} \}_{j=1}^{M}$ & Epsilon-net of mesh size $r$ satisfying \eqref{0726-b-1} \\ \hline
    $y_{l} = y^{j_{l}} $ & One of the nearest points from $x_{l}$ among $G$ defined below \eqref{0726-b-1}\\ \hline
    $R^{l}$ & SRW started at $y_{l}$ considered in \eqref{0726-c-1} \\ \hline
    ${\cal U}^{1,l}_{n}$ & Subtree generated in Wilson's algorithm as described in \eqref{0726-c-1} \\ \hline
    $t^{l}$ &  First time that $R^{l}$ hits ${\cal U}^{1,l-1}_{n} \cup \partial D_{n}$ defined in \eqref{0726-c-1}\\ \hline 
    $w_{l} $ &  $R^{l} ( t^{l} )$ defined above \eqref{0726-h} \\ \hline 
    $H$, $I$ and $K_n$ & Events as defined in \eqref{0726-h}, \eqref{0726-i} and \eqref{0630-w} resp. \\ \hline
    $\gamma_{x, \partial}$ & Unique simple path in ${\cal U}_{n}$ from $x$ to $\partial D_{n}$ defined above \eqref{Udef}\\ \hline
    $U_{x}=T^{\gamma_{x, \partial}}_{x,\sqrt{r}}$ &  First time that $\gamma_{x, \partial}$ exits from $B (x, \sqrt{r} )$ defined in \eqref{Udef}\\ 
   \hline \hline
  \end{tabular}
  \caption{List of notation used in Section \ref{sec:3}.}
  \label{symbols-3}
\end{table}

Recall the definition of $K_n$ in \eqref{0630-w}. 
Since $\eta_{n} (t) = \gamma_{n} \big( 2^{\beta n} t \big)$ and $\eta_{n}$ converges to $\eta$ in distribution in $\rho$-metric, we have that 
\begin{equation}\label{0626-d-2}
P (\widetilde{A}(\delta))  \le C\delta^{c}.
\end{equation}

Now we are ready to prove \eqref{0626-d-1}. Applying the inequality \eqref{0626-d-2} to the case that $\delta = 2^{-m}$ with $2^{\frac{\epsilon m}{2}} > 100$ and in view of the Borel-Cantelli lemma,  it holds that with probability one, there exists $m_{1} < \infty$ such that the event $\widetilde{A}(2^{-m})$ holds for all $m \ge m_{1}$. 

Take $0 \le s < t \le t_{\eta}$. If $t- s \ge 2^{-m_{1}}$, we see that 
\begin{equation*}
 \frac{|\eta (s) - \eta (t) |}{|s-t|^{h}} \le 2^{m_{1} + 1}.
\end{equation*}
On the other hand, if $0 < t-s < 2^{-m_{1}}$, we can find some $m \ge m_{1}$ such that $ 2^{-m -1} \le t-s < 2^{-m}$. Since $\widetilde{A}(2^{-m})$ holds for this $m$,  we have that there exist $C=C(h)$ and $C'=C'(h)$ such that 
\begin{equation*}
  |\eta (s) - \eta (t) | \le 2^{- m h } \le C (t-s)^{h} \quad \mbox{ and }\quad  \frac{|\eta (s) - \eta (t) |}{|s-t|^{h}} \le C'.
\end{equation*}
Taking supremum, we see that \eqref{0626-d-1} holds for $h$ defined in \eqref{hdef}. 
To see that it holds simultaneously for all $h<1/\beta$, take $h=1/\beta-2^{-k}$ for integer $k\geq 4$ and let $k\to\infty$.
\end{proof}

\section{The critical case}\label{sec:4}
In this section, we prove the second half of Theorem \ref{main}, namely that almost surely the scaling limit of $\{ \eta_{n} \}_{n \ge 1}$ is not ${1}/{\beta}$-H\"older continuous. See  Proposition \ref{CRITICAL} below.

We give a brief explanation of the proof here. As illustrated in Figure \ref{trav-f}, we consider a sequence of cubes $\{ Q_{i} \}_{i=0}^{M}$ of side length $2^{-m}$. The center of the first cube $Q_{0}$ is the origin from which the LERW $\gamma_{n}$ starts.
We are then interested in the type of event (which we refer to as ``the event $A$'') on which, heuristically speaking, the following conditions are fulfilled:
\begin{itemize}
\item The LERW $\gamma_{n}$ keeps going to the right until it exits from the union of $Q_{i}$ (we call this union a ``tube'') with no big backtracking as described in Figure \ref{trav-f};  

\item The number of points in $Q_{i} \cap 2^{-n} \mathbb{Z}^{3}$ hit by $\gamma_{n}$ is bounded above by $C 2^{- \beta m} 2^{\beta n}$ for all $0 \le i \le M$. 
\end{itemize}
See \eqref{am} for the precise definition of the actual event $A$ we will be working with in this section. Choosing the constant $C < \infty$ appropriately, we will give a lower bound on the probability of such events in Proposition \ref{lem-critical}.
Note that if the event $A$ occurs, $\gamma_{n}$ exits from $B \big( M 2^{-m} \big)$ within $C M 2^{- \beta m} 2^{\beta n}$ steps, which is much smaller than the typical number of steps for $\gamma_{n}$ to exit from that ball (the typical number of steps is comparable to $M^{\beta} 2^{- \beta m} 2^{\beta n}$).
Consequently, we have a desired upper bound on modulus of continuity of $\gamma_{n}$ restricted in the the tube
on the event $A$; see Proposition \ref{cor-cri-1} for precise statements.
 
To conclude with an a.s.\ statement, we will employ the following iteration argument to which is similar to that used in  \cite[Proposition 6.6]{BM} and \cite[Proposition 8.11]{S}. Namely, as illustrated in Figure \ref{iteration-f}, we prepare a sequence of concentric boxes $\{ B^{l} \}_{l =1}^{M'}$ with some large $M'$ and check whether for each $l$ the LERW $\gamma_{n}$ restricted in a tube $Q$ contained in the annulus $B^{l} \setminus B^{l-1}$ travels a long distance in a short time as discussed in the previous paragraph, which can be described by a translate of the event $A$. Conditioning $\gamma_{n}$ up to the first time that it exits from $B^{l-1}$,
we will give a lower bound on the (conditional) probability that $\gamma_{n}$ restricted in $Q$ speeds up as above (see Proposition \ref{lem-con-2} below for this). 
Taking $M'$ (the number of concentric boxes) sufficiently large and iterate the argument above, we can show that the scaling limit $\eta$ is not $1/\beta$-H\"older continuous in Proposition \ref{CRITICAL} at the end of this section. 

Throughout this section, we often omit the superscript $(n)$ when we write the discrete ball $B^{(n)} (x, r)$.    

\vspace{3mm}

\begin{figure}[h]
\begin{center}
\includegraphics[scale=0.7]{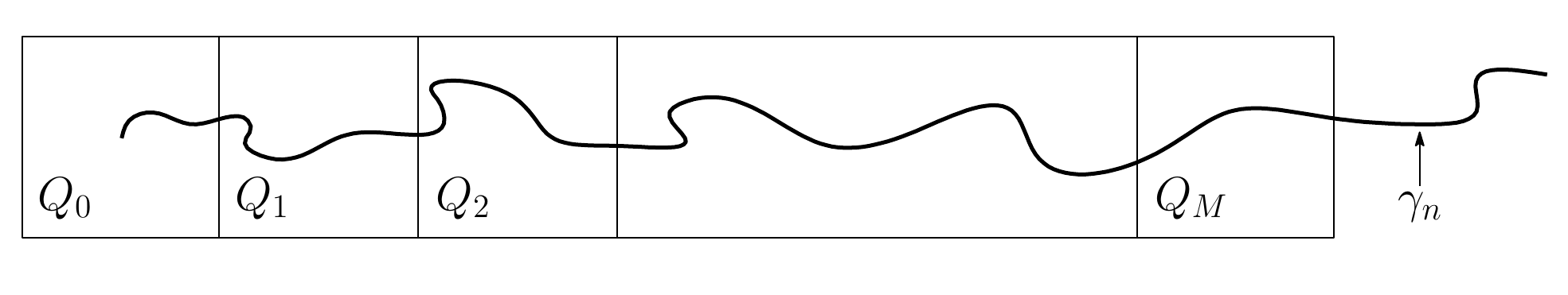}
\caption{Illustration for the event $A$.}\label{trav-f}
\end{center}
\end{figure}

\vspace{3mm}
\begin{figure}[h]
\begin{center}
\includegraphics[scale=0.7]{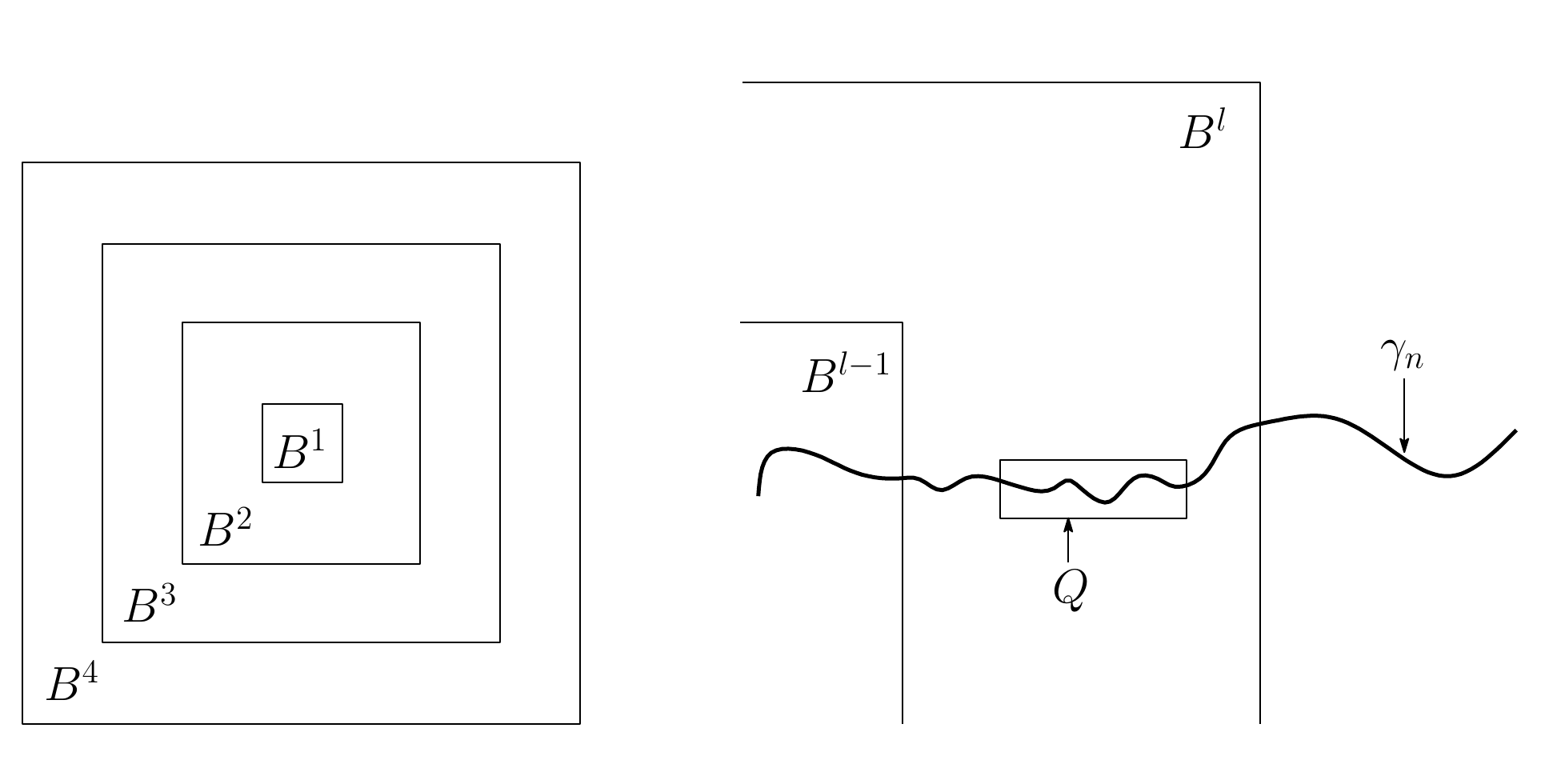}
\caption{Illustration for an iteration argument. The LERW $\gamma_{n}$ speeds up in the tube $Q$ for each $l$.}\label{iteration-f}
\end{center}
\end{figure}

\subsection{Tube-crossing  estimates}
In this subsection, we establish estimates on the probability that a LERW stays in a tube and travels atypically fast in Proposition \ref{lem-critical} and show that with positive probability a LERW is not $1/\beta$-H\"older continuous, which is summarized in Proposition \ref{cor-cri-1}.

We begin with our setup.
Let 
\begin{equation}\label{eq:mdef}
m\in \mathbb{N}\mbox{ such that $m_0:=m^{\frac{1}{10}} \in [10,\infty)\cap \mathbb{N}$.}
\end{equation}
\begin{dfn}\label{partition}
Define a collection of cubes $\{ Q_{i} \}_{i \ge 0} = \{ Q_{i}^{m} \}_{i \ge 0}$  by 
\begin{equation}
Q_{i}^{m} = \big\{ x \in \mathbb{R}^{3} \ \big| \ \| x -  x_{i} \|_{\infty} \le 2^{-m} \big\}  \ \ \ \ \text{ where }  \ \ \ \ x_{i} = \big( i \cdot 2^{-m + 1}, 0, 0 \big) \ \ \text{ for } \ \ i \ge 0.
\end{equation}
Namely, $Q_{i}^{m} $ stands for the cube of side length $2^{-m +1}$ centered at $x_{i}$. 

For $a < b$, we write 
\begin{align*}
&H [a, b]  = \big\{ x = (x^{1}, x^{2}, x^{3}) \in \mathbb{R}^{3} \ \big| \  a \le x^{1} \le b, \ -2^{-m} \le x^{2}, \, x^{3} \le 2^{-m} \big\}, \notag \\
&H (a, b)  = \big\{ x = (x^{1}, x^{2}, x^{3}) \in \mathbb{R}^{3} \ \big| \  a < x^{1} < b, \ -2^{-m} \le x^{2}, \, 
x^{3} \le 2^{-m} \big\}.
\end{align*}
We define $H (a, b] $ and $H [a, b)$ similarly. Also, we set 
\begin{align}\label{0902-d-1}
&H (a) =  \big\{ x = (x^{1}, x^{2}, x^{3}) \in \mathbb{R}^{3} \ \big| \   x^{1} = a , \ -2^{-m} \le x^{2}, \, x^{3} \le 2^{-m} \big\}, \notag \\
&G (a) =  \big\{ x = (x^{1}, x^{2}, x^{3}) \in \mathbb{R}^{3} \ \big| \   x^{1} = a , \ -2^{-m-1} \le x^{2}, \, x^{3} \le 
2^{-m-1} \big\}.
\end{align}
\end{dfn}

We regard $H (a)$ as a ``partition'' of $H ( - \infty, \infty )$. By construction, we have $G (a) \subset  H (a)$. 
We also note that writing 
\begin{equation}\label{a_i}
a_{i} = 2^{-m} (2i - 1 ),
\end{equation}
it holds that $Q_{i}^{m} = H [a_{i}, a_{i+1}]$. Thus, $H (a_{i} ) $ (resp. $H (a_{i+1} )$) 
coincides with the left (resp. right) face of $Q_{i}^{m}$.

Recall that $S=S^{(n)}$ stands for the SRW on $2^{-n} \mathbb{Z}^{3}$ started at the origin. Assume that $$2^{-n}<\exp \big\{ - 2^{m^{100}} \big\}.$$ 
By linear interpolation, we may assume that $S (k)$ is defined for every non-negative real $k$ and that $S[0, \infty)$ is a continuous curve.

For a continuous curve $\lambda$ lying in $\mathbb{R}^{3}$, we define a family of hitting times $\{ t_{\lambda} (a) \}_{a \in \mathbb{R} }$ by 
\begin{align}\label{t-lambda}
t_{\lambda} (a) = \inf  \big\{ k \ge 0 \ \big| \ \lambda (k) \in H (a)  \big\},
\end{align}
with the convention that $\inf \emptyset = + \infty$. Let 
\begin{equation}\label{eq:qdef}
q = 2^{- m - m_{0} }.
\end{equation}

We then define the events $A_{i}$ as follows
\begin{align}\label{0907-1}
&A_{0} = \Big\{ t_{S} (a_{1} ) < \infty, \  S \big( t_{S} ( a_{1} ) \big) \in G ( a_{1} ), \ S \big[ 0, t_{S} (a_{1} ) \big] \subset  Q^{m}_{0}, \notag\\
&\qquad\qquad\qquad\qquad\qquad\qquad\qquad\qquad\ S \big[ t_{S} ( a_{1} - q  ), t_{S} (a_{1} ) \big] \cap H \big( a_{1} - 2 q \big) = \emptyset \Big\}, \notag \\
&\text{and for }i\geq 1,\\
&A_{i} = \Big\{ t_{S} (a_{i} ) < t_{S} (a_{i+1} ) < \infty, \  S \big( t_{S} ( a_{i+1} ) \big) \in G ( a_{i+1} ),  \notag \\
& \ \ \ \ S \big[ t_{S} (a_{i}), t_{S} (a_{i+1} )  \big] \subset  H (a_{i} -q, a_{i+1} ], S \big[ t_{S} ( a_{i+1} - q  ), t_{S} (a_{i+1} ) \big] \subset H \big[ a_{i+1} - 2 q, a_{i+1} \big]   \Big\} .\notag
\end{align}
In other words, the event $A_{0}$ ensures that the SRW exits $Q^{m}_{0}$ from $G (a_{1} )$ and excludes excessive backtracking of $S$ near time $t_S(a_1)$. 
Similarly, the event $A_{i}$ also guarantees that the SRW hits $H (a_{i+1})$ in $G ( a_{i+1} )$ and  excludes excessive backtracking of $S$, near both the beginning and the end of the time interval $\big[ t_{S} (a_{i}), t_{S} (a_{i+1} )  \big]$.
We note that the event $A_{i}$ is measurable with respect to $S \big[ t_{S} (a_{i}), t_{S} (a_{i+1} )  \big]$ for $i \ge 1$, while $A_{0}$ is an event measurable with respect to $S \big[ 0, t_{S} (a_{1} ) \big]$. We set 
\begin{equation}\label{0907-b-1}
F_{i} = \bigcap_{k=0}^{i} A_{k}.
\end{equation}
We observe that on the event $F_i$ it holds that $0 \le t_{S} (a) < t_{S} (a' ) \le t_{S} (a_{i+1}) < \infty$ for all $0 \le a < a' \le a_{i+1}$.

\medskip

Given a path $\lambda:[0,t_\lambda]\to 2^{-n} \mathbb{Z}^3$, a time $t\in[0,t_\lambda]\cap \mathbb{Z}$ is called a cut time for $\lambda$, if $\lambda[0,t]\cap\lambda[t+1,t_\lambda]=\emptyset$. In this case, $\lambda(t)$ is called a cut point for $\lambda$.

We next consider a special type of cut time defined on the event $A_{i}$.

\begin{dfn}\label{special-cut}
Suppose that $A_{i}$ defined in \eqref{0907-1} occurs. For each $i \ge 1$, we say  $k$ is a {\bf nice cut time} or $S$ has a nice cut point in $Q_{i}^{m}$ if the following four conditions are fulfilled:
\begin{itemize}
\item[(i)] $t_{S} \big( a_{i} + \frac{q}{2} \big) \le k \le t_{S} ( a_{i} + q )$,

\item[(ii)] $S \big[ t_{S} (a_{i} ), k \big] \cap S \big[ k+1, t_{S} ( a_{i +1 } ) ] = \emptyset $,

\item[(iii)] $S (k) \in H \big[ a_{i} + \frac{q}{2}, a_{i} + q \big] $,

\item[(iv)] $ S [k, t_{S} ( a_{i+1} ) ] \cap H (a_{i} ) = \emptyset $.

\end{itemize}
See Figure \ref{CT-f}. 
\end{dfn}

\begin{figure}[h]
\begin{center}
\includegraphics[scale=1.0]{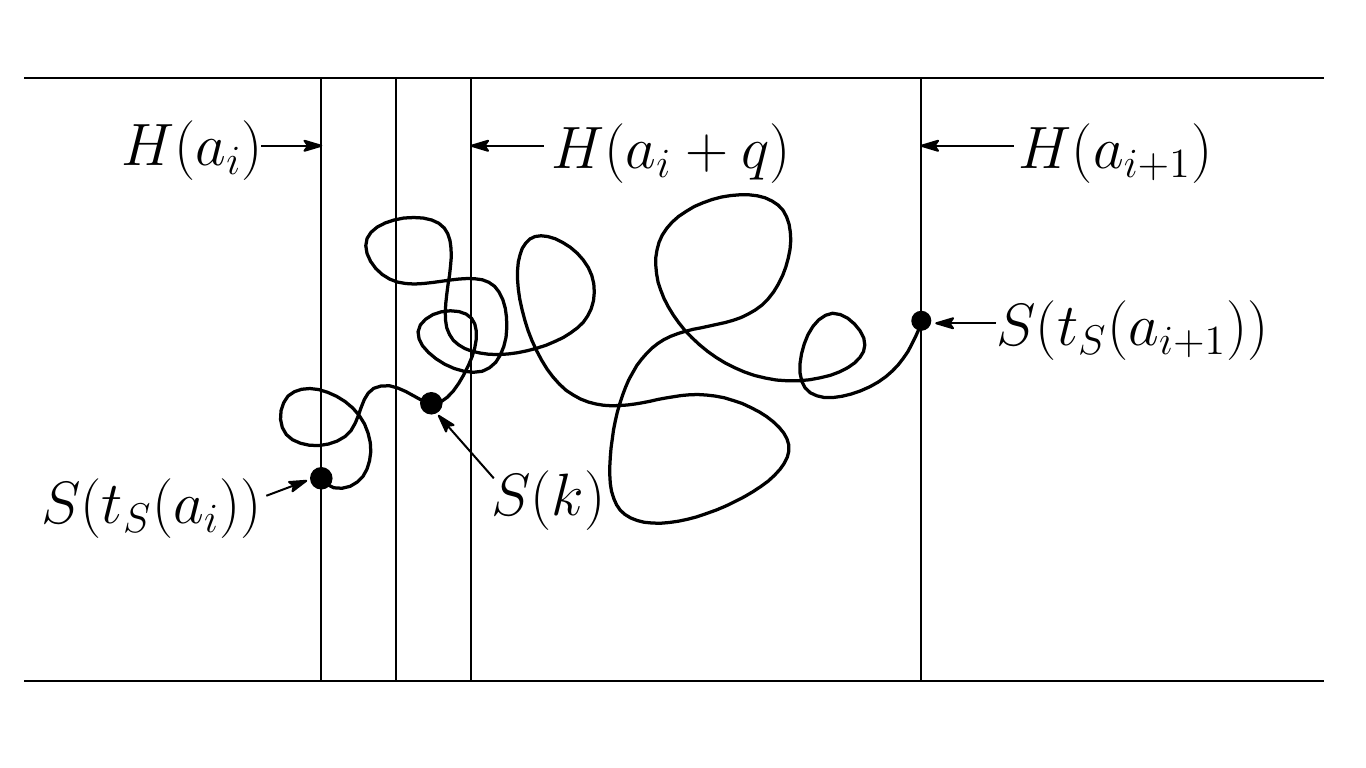}
\caption{Illustration for a nice cut point $S (k)$ in $Q_{i}^{m}$. The vertical line between $H (a_{i})$ and $H (a_{i} + q)$ stands for $H (a_{i} + \frac{q}{2} )$.}\label{CT-f}
\end{center}
\end{figure}

The event $B_{i}$ is then defined as
\begin{equation}\label{Bi}
B_{i} = \big\{ S\text{ has a nice cut point in } Q_{i}^{m} \big\}  \ \text{ for } \ i \ge 1.
\end{equation}
We note that the event $B_{i}$ is measurable with respect to $S \big[ t_{S} (a_{i}), t_{S} (a_{i+1} )  \big]$.
We define
\begin{equation}\label{0907-b-2}
G_{i} = \bigcap_{k=1}^{i} B_{k}.
\end{equation}

We set 
\begin{align}\label{09070}
&\xi_{i} = \text{LE} \big( S \big[ 0, t_{S} (a_{i+1} ) \big] \big), \  \   \ \  \ \  \lambda_{i} =  \text{LE} \big( S \big[ t_{S} (a_{i}), t_{S} (a_{i+1} ) \big] \big) \ \ \  \text{ for } \  i \ge 0, \notag \\ 
&\ \ \ \text{ and } \ \xi_{0}' = \xi_{0} \ \ \ \   \ \ \ \ \xi_{i}' = \xi_{0} \oplus \lambda_{1} \oplus \cdots \oplus \lambda_{i} \  \text{ for }  \  i \ge 1.
\end{align}
Note that $\xi_{i}'$ is not necessarily a simple curve, and thus $\xi_{i} \neq \xi_{i}'$ in general. However, in the next lemma, we show that one can bound the length of $\xi_{i}$ by that of $\xi_{i}'$ on the event $F_{i} \cap G_{i}$ defined in \eqref{0907-b-1} and \eqref{0907-b-2} respectively.

\begin{lem}\label{lem1-critical}
Let $i \ge 1$. Suppose that the event $F_{i} \cap G_{i}$ occurs. Then, we have 
\begin{equation}\label{0902'}
{\rm len} ( \xi_{i} ) \le   {\rm len} (\xi_{0} ) + \sum_{l = 1}^{i} \Big\{ {\rm len} ( \lambda_{l} ) + \Big| \xi_{i} \cap H \big( a_{l} - q, a_{l} + q \big]  \Big| \Big\}={\rm len} ( \xi_{i}' ) +\sum_{l = 1}^{i} \Big\{  \Big| \xi_{i} \cap H \big( a_{l} - q, a_{l} + q \big]  \Big| \Big\},
\end{equation}
where for $D \subset \mathbb{R}^{3}$ we write $|D| $ for the number of points in $D \cap 2^{-n} \mathbb{Z}^{3}$. 
\end{lem}

\begin{proof} The equality in the claim follows from definition, hence we only prove the inequality. 

To compare $\xi_{i}$ and $\xi_{i}'$ on the  event $F_{i} \cap G_{i}$, we will first consider the case that $i=1$. Suppose that 
the event $F_{1} \cap G_{1}$ occurs. By definition of the loop-erasing procedure, $\xi_{1}$ is generated in the following way. Let 
\begin{equation*}
s_{1} = \inf \big\{ k \ge 0 \ \big| \ \xi_{0} (k) \in S \big[ t_{S} (a_{1}), t_{S} ( a_{2} ) \big] \big\} 
\end{equation*}
 and 
\begin{equation*}   \  t_{1} = \sup \big\{  k\in [t_{S} (a_{1}) , t_{S} (a_{2} )] \ \big| \ S(k) = \xi_{0} (s_{1}) \big\}.
\end{equation*}
Then it follows that 
\begin{equation*}
\xi_{1} = \xi_{0} [0, s_{1}]  \oplus  \text{LE} \big( S \big[ t_{1}, t_{S} (a_{2} ) \big] \big).
\end{equation*}
Since $S \big[ t_{S} ( a_{1} ), t_{S} ( a_{2} ) \big] \cap H \big( a_{1} - q \big) = \emptyset$ by the event $A_{1}$, we have that 
\begin{equation*}
\xi_{0} (s_{1} ) \in H \big( a_{1}- q, a_{1} \big]   \ \ \text{ and }  \ \             t_{\xi_{0}} \big( a_{1} - q \big) \le s_{1} \le t_{\xi_{0}} (a_{1} ).
\end{equation*}
Let $k_{1}$ be a nice cut time in $Q_{1}^{m}$. By Definition \ref{special-cut}(iv), it follows that $t_{1} \le k_{1}$ (because otherwise $S (t_{1}) \in S \big[ k_{1}, t_{S} (a_{2} ) \big] \cap Q_{0}^{m} $ and thus $S \big[ k_{1}, t_{S} (a_{2} ) \big] \cap H (a_{1}) \neq \emptyset $, which contradicts Definition \ref{special-cut}(iv) above). As a cut point of a path must belong to its loop-erasure, one can decompose $\text{LE} ( S [ t_{1}, t_{S} (a_{2} ) ] )$ as follows:
\begin{equation*}
\text{LE} \big( S \big[ t_{1}, t_{S} (a_{2} ) \big] \big) = \text{LE} \big( S \big[ t_{1}, k_{1} \big] \big) \oplus \text{LE} \big( S \big[ k_{1}, t_{S} (a_{2} ) \big] \big).
\end{equation*}
Similarly, it follows that  
\begin{equation*}
\lambda_{1} = \text{LE} \big( S \big[ t_{S} (a_{1}) , k_{1} \big] \big) \oplus \text{LE} \big( S \big[ k_{1}, t_{S} (a_{2} ) \big] \big).
\end{equation*}
Definition \ref{special-cut}(i) guarantees that $k_{1} \le t_{S} (a_{1} + q ) $. Combining this with the condition $S \big[ t_{S} ( a_{1} ), t_{S} ( a_{2} ) \big] \cap H \big( a_{1} - q \big) = \emptyset$ in the event $A_{1}$, we see that 
\begin{equation*}
S \big[ t_{1}, k_{1} \big]   \subset  H \big( a_{1} - q, a_{1} + q \big].
\end{equation*}
Consequently, it follows that 
\begin{align}\label{0902-1}
\xi_{1} &=  \xi_{0} [0, s_{1}]  \oplus  \text{LE} \big( S \big[ t_{1}, t_{S} (a_{2} ) \big] \big)  \notag \\
&= \xi_{0} [0, s_{1}]  \oplus   \text{LE} \big( S \big[ t_{1}, k_{1} \big] \big) \oplus \text{LE} \big( S \big[ k_{1}, t_{S} (a_{2} ) \big] \big) \notag \\
&\subset \xi_{0} \cup \text{LE} \big( S \big[ t_{1}, k_{1} \big] \big) \cup \lambda_{1} \notag \\
&\subset \xi_{0} \cup \big( \xi_{1} \cap H \big( a_{1} - q, a_{1} + q \big] \big) \cup \lambda_{1}.  
\end{align}

We next deal with the case of general $i$'s. On the event $F_{i} \cap G_{i}$, let 
\begin{equation*}
s_{i} = \inf \big\{ k \ge 0 \ \big| \ \xi_{i-1} (k) \in S \big[ t_{S} (a_{i}), t_{S} ( a_{i+1} ) \big] \big\} 
\end{equation*}
 and 
 \begin{equation*}   t_{i} = \sup \big\{ k\in[ t_{S} (a_{i}) , t_{S} (a_{i+1} )] \ \big| \ S(k) = \xi_{i-1} (s_{i}) \big\}.
\end{equation*}
A similar argument as above shows that on the event $F_{i} \cap G_{i}$ we have
\begin{align}\label{0902-2}
&\xi_{i} = \xi_{i-1} [0, s_{i} ] \oplus \text{LE} \big( S \big[ t_{i}, k_{i} \big] \big) \oplus \text{LE} \big( S \big[ k_{i}, t_{S} (a_{i+1} ) \big] \big), \notag \\
&\lambda_{i} = \text{LE} \big( S \big[ t_{S} (a_{i}) , k_{i} \big] \big) \oplus \text{LE} \big( S \big[ k_{i}, t_{S} (a_{i+1} ) \big] \big) \ \ \text{ and }  \  \  S [ t_{i},  k_{i} ] \subset H \big[ a_{i} - q, a_{i} + q \big], 
\end{align}
where $k_{i}$ is a nice cut time such that $S(k_i)\in Q_{i}^{m}$.

Now we are ready to show \eqref{0902'} by induction.
When $i =1$, the inequality \eqref{0902'} immediately follows from \eqref{0902-1}. Suppose that the inequality \eqref{0902'} holds for $i -1$ and that the event $F_{i} \cap G_{i}$ occurs. Using this assumption and \eqref{0902-2}, we see that
$$
\xi_i \subseteq \xi_{i-1} \cup (\xi_{i} \cap H ( a_{i} - q, a_{i} + q \big]) \cup \lambda_i,
$$
hence,
\begin{align}\label{0902-c-2}
& {\rm len} ( \xi_{i} ) \le {\rm len} ( \xi_{i-1} ) + \Big| \xi_{i} \cap H \big( a_{i} - q, a_{i} + q \big]  \Big| + {\rm len} (\lambda_{i}) \notag \\
 \le &\;{\rm len} (\xi_{0} ) + \sum_{l = 1}^{i-1} \Big\{ {\rm len} ( \lambda_{l} ) + \Big| \xi_{i-1} \cap H \big(a_{l} - q, a_{l} + q \big]  \Big| \Big\} +  \Big| \xi_{i} \cap H \big( a_{i} - q, a_{i} + q \big]  \Big| + {\rm len} (\lambda_{i}).
 \end{align}

However, we claim that for each $1 \le l \le i-1$, 
\begin{equation}\label{0902-a-1}
\Big| \xi_{i-1} \cap H \big( a_{l} - q, a_{l} + q \big]  \Big| = \Big| \xi_{i} \cap H \big( a_{l} - q, a_{l} + q \big]  \Big|.
\end{equation}
The reason is that by the fourth condition of $A_{i}$, we see that $S \big[ t_{i},   t_{S} (a_{i+1} ) \big] \cap H \big( a_{l} - q, a_{l} + q \big] = \emptyset $ for $1 \le l \le i-1$ (note that $a_{i-1} + q < a_{i} - q$ since $m_{0} \ge 10$).  
Using \eqref{0902-a-1} and the first equation in \eqref{0902-2}, we have that 
\begin{equation}\label{0902-c-1}
\xi_{i} \cap H \big( a_{l} - q, a_{l} + q \big] = \xi_{i-1} [0, s_{i} ] \cap H \big( a_{l} - q, a_{l} + q \big].
\end{equation}
On the other hand, since $\xi_{i-1} (s_{i} ) = S (t_{i} )  \in H \big( a_{i} - q, a_{i} \big]$, we claim that
\begin{equation}\label{0902-b-1}
\xi_{i-1} \big[ s_{i}, {\rm len} (\xi_{i-1}) \big] \subset H \big[ a_{i} - 2q , a_{i} \big].
\end{equation}
We will prove \eqref{0902-b-1} by contradiction. Suppose that \eqref{0902-b-1} does not hold. This implies that  $\xi_{i-1}$ hits $H (a_{i}- 2 q) \neq \emptyset$ after time $s_{i}$. However, because $\xi_{i-1} (s_{i} )  \in H \big[ a_{i} - q, a_{i} \big]$, 
we conclude that $S \big[ t_{S} ( a_{i} - q), t_{S} ( a_{i} ) \big] \cap H (a_{i} - 2 q ) \neq \emptyset$. But this 
contradicts the last condition of the event $A_{i-1}$. So \eqref{0902-b-1} must hold. It follows from $m_{0} \ge 10$ again that $H \big[ a_{i} - 2q , a_{i} \big] \cap H \big[ a_{l} - q, a_{l} + q \big]  = \emptyset $ for all  $1 \le l \le i-1$. Consequently, we have 
\begin{equation*}
\xi_{i-1} \cap H \big[ a_{l} - q, a_{l} + q \big] = \xi_{i-1} [0, s_{i} ] \cap H \big[ a_{l} - q, a_{l} + q \big] 
\end{equation*}
for each $1 \le l \le i-1$. Combining this with \eqref{0902-c-1}, we get \eqref{0902-a-1}. Then the inequality \eqref{0902'} for $i$ follows from \eqref{0902-c-2} and \eqref{0902-a-1}.
\end{proof}

We will next deal with the length of each $\lambda_{i}$ defined as in \eqref{09070}. For $C \ge 1$, define the event $D_{i} (C)$ by
\begin{equation}\label{c}
D_{0} = D_{0} (C) = \big\{ {\rm len} ( \xi_{0} ) \le C 2^{- \beta m} 2^{\beta n} \big\} \ \ \text{ and }  \ \ D_{i} = D_{i} (C) =  \big\{ {\rm len} (\lambda_{i} )  \le C 2^{- \beta m } 2^{\beta n} \big\} 
\end{equation}
for $ i \ge 1$, where $\xi_{0}$ is as defined in \eqref{09070}.
We also recall that $G (a)$ is defined as in \eqref{0902-d-1}. The next lemma gives a lower bound on the probability of $A_{i} \cap B_{i} \cap D_{i} (C)$ choosing $C$ sufficiently large, where the event $A_{i}$ and $B_{i}$ are given as in \eqref{0907-1} and \eqref{Bi}.

\begin{lem}\label{lem2-critical}
 There exist universal constants $0 < c_{\ast}, C_{\ast} < \infty$ such that
\begin{equation}\label{0902-d-2}
P \big( A_{0} \cap D_{0} (C_{\ast} ) \big ) \ge c_{\ast} \ \ \text{ and }  \ \  \min_{x \in G (a_{i}) } P^{x} \big( A_{i} \cap B_{i} \cap D_{i} (C_{\ast}) \big) \ge c_{\ast} 2^{- m_{0}}  \ \ \ \forall i \ge 1.
 \end{equation}
 \end{lem}
Note that if we assume $S (0) = x \in G (a_{i} )$, we must have $t_{S} (a_{i} ) = 0$.

Before giving the proof of this lemma, we make a quick detour to recall a result of Lawler on the existence of local cut points of 3D SRW. 
\begin{dfn}\label{localnice}
We say $k$ is a {\bf local} nice cut time or $S$ has a local nice cut point in $Q_{1}^{m}$ if the following four conditions (i') - (iv') are fulfilled:
\begin{itemize}
\item[(i')] $t_{S} \big( a_{1} + \frac{q}{2} \big) \le k \le t_{S} ( a_{1} + q )$,

\item[(ii')] $S \big[ t_{S} (a_{1} ), k \big] \cap S \big[ k+1, t_{S} ( a_{1 } + 2q ) ] = \emptyset $ and $S [t_{S} ( a_{1} ), k] \cap H (a_{1} - q ) = \emptyset$,

\item[(iii')] $S (k) \in H \big[ a_{1} + \frac{q}{2}, a_{1} + q \big] $,

\item[(iv')] $ S [k, t_{S} ( a_{1} + 2q ) ] \cap H (a_{1} ) = \emptyset $ and $S [t_{S} ( a_{1} ), t_{S} ( a_{1} + 2q ) ]  \subset B \Big( S \big( t_{S} (a_{1} ) \big), 4q  \Big) $.

\end{itemize}
\end{dfn}
Compare this with a nice cut time as defined in Definition \ref{special-cut}. The difference between them comes from the second and fourth conditions. We write 
$$
B_{1}':=\{\mbox{$S$ has a local nice cut point in $Q_{1}^{m}$}\}.
$$
Note that if $S$ has a local cut point in $Q_{1}^{m}$ and $S \big[ t_{S} (a_{1} + 2q), t_{S} (a_{2} ) \big]$ is contained in $H ( a_{1} + q, a_{2} ]$, then $S$ has a nice cut point in $Q_{1}^{m}$.

By adapting the argument of \cite[Corollary 5.2]{Lawcut}, in which an essentially similar result is proved, it is possible to check that the probability $S$ has a local nice cut point in $Q_{1}^{m}$ is uniformly bounded from below. More precisely, 
\begin{equation}\label{1003-1-1}
P^{x} ( B_{1}' ) \ge c_{3}'  \ \ \ \ \text{uniformly in $x \in G (a_{1} ),$}
\end{equation}
for some universal constant $0< c_{3}' < \infty$.

\begin{proof}

The first assertion of \eqref{0902-d-2} is proved as follows. It follows from easy estimates on harmonic measures (to do this, one can for example further require that $S(t_{S}(a_1-q))\in G(a_1-q)$ and then apply strong Markov property) that $P (A_{0} ) \ge c' > 0$. Moreover, we already know (see \cite[Theorem 8.4]{S} and \cite[Corollary 1.3]{Escape}) that $E \big( {\rm len} ( \xi_{0} ) \big) \asymp 2^{- \beta m} 2^{\beta n} $. Thus, taking $C_{\ast}$ sufficiently large and using Markov's inequality, we have $P \big[ \big( D_{0} (C_{\ast} ) \big)^{c} \big] \le \frac{c'}{2}$. This implies the first assertion of \eqref{0902-d-2} with $c_{\ast} = \frac{c'}{2}$.

We will next consider the second assertion of \eqref{0902-d-2}. Thanks to translation invariance, it suffices to consider the case that $i=1$ for the second assertion of \eqref{0902-d-2}.
Take $x \in G (a_{1} )$. We first estimate the probability of $A_{1}$. 
Using the gambler's ruin estimate (see \eqref{Gambler} for this), it holds that 
\begin{equation*}
c_{1} 2^{-m_{0}} \le  P^{x} (A_{1} ) \le c_{2} 2^{-m_{0} }  \ \ \ \ \text{uniformly in $x \in G (a_{1} ),$} 
\end{equation*}
for some universal constants $0 < c_{1} , c_{2} < \infty$. Thus, in order to prove the second assertion of \eqref{0902-d-2}, it suffices to show that 
\begin{equation}\label{0902-e-1}
P^{x} \big( B_{1} \cap D_{1} (C) \ \big| \ A_{1} \big) \ge c  \ \ \ \ \text{uniformly in $x \in G (a_{1} ),$}
\end{equation}
for some universal constants $0< c, C < \infty$.

 However, using the strong Markov property at time $t_{S} (a_{1} + 2q)$, the gambler's ruin estimate (see \eqref{Gambler} for this), properties of local nice cut time and cut points (see Definition \ref{localnice}), and \eqref{1003-1-1}, it follows that  
\begin{align*}
 &\quad \; \, P^{x} (A_{1} \cap B_{1} ) \ge  P^{x} \Big( A_{1}, \ B_{1}', \  S \big[ t_{S} (a_{1} + 2q), t_{S} (a_{2} ) \big] \subset H ( a_{1} + q, a_{2} ] \Big) \\
 &\ge P^{x} (B_{1}') \min_{z \in H (a_{1} + 2 q) \cap B (x, 4q ) } P^{z} \Big( t_{S} (a_{2} ) < \infty, \ S (t_{S} (a_{2}) ) \in G (a_{2} ), \\ 
 & \quad \quad \  \ \ \ \ \ \ \ \ \ \ \ \ \ \ \ \ \ \ \ \ \ \ \ \ \   S [0, t_{S} (a_{2} ) ]  \subset H ( a_{1} + q, a_{2} ],\   S[ t_{S} (a_{2} -q), t_{S} (a_{2} ) ] \subset H [a_{2}- 2q, a_{2} ] \Big) \\
 & \ge c_{3} P^{x} (A_{1} )  \  \ \ \ \ \ \ \ \text{uniformly in $x \in G (a_{1} ),$}
\end{align*}
for some universal constant $0< c_{3} < \infty$. Here we used the gambler's ruin estimate in the last inequality above to show that 
\begin{equation*}
P^{z} \Big( S [0, t_{S} (a_{2} ) ]  \subset H ( a_{1} + q, a_{2} ] \Big) \ge c 2^{-m_{0}} 
\end{equation*}
uniformly in $z \in H (a_{1} + 2 q) \cap B (x, 4q )$  and $x \in G (a_{1} )$.
Thus, we have 
\begin{equation*}
P^{x} \big( B_{1}  \ \big| \ A_{1} \big) \ge c_{3}  \ \ 
\end{equation*}
uniformly in $x \in G (a_{1} )$.

 The proof of the lemma is completed once we show that choosing the constant $C_{1} < \infty$ appropriately, 
\begin{equation}\label{0902-f-1}
 P^{x} \big(  \big( D_{1} (C_{1}) \big)^{c} \ \big| \ A_{1} \big) \le \frac{c_{3}}{2}  \ \ \ \ \text{uniformly in $x \in G (a_{1} ).$}
\end{equation}
To prove this, we consider a conditioned random walk $Y$ in $2^{-n} \mathbb{Z}^{3}$ on the event $A_{1}$, which starts from $x \in G (a_{1} )$. We denote its law and respective expectation by $P^x$ and $E^x$ respectively. Let $ \iota = \text{LE} \big( Y \big[ 0, t_{Y} (a_{2}) \big] \big) $. We will give an upper bound on $E^{x} \big( {\rm len} (\iota) \big)$. To do it, let (recall the definition of $q$ in \eqref{eq:qdef})
\begin{equation*}
d_{y} = |x- y| \wedge \text{dist} \big( y,  \partial H [a_{1} -q, a_{2} ] \big) \ \ \text{ for } \ y \in H [a_{1} -q, a_{2} ].
\end{equation*}
Then Proposition \ref{prop:npoint2} ensures that 
\begin{equation*}
P^{x} ( y \in \iota ) \le C G_{H [a_{1} -q, a_{2} ]} \big( x, y; Y \big) \text{Es} \big( 2^{n} d_{y} \big) \ \ \text{ for } \ y \in H [a_{1} -q, a_{2} ],
\end{equation*}
where 
\begin{equation*}
G_{H [a_{1} -q, a_{2} ]} \big( x, y; Y \big) = E^{x} \bigg( \sum_{j=0}^{t_{Y} (a_{2})} {\bf 1} \big\{ y \in Y (j) \big\} \bigg) 
\end{equation*}
stands for the Green's function of $Y$ in $H [a_{1} -q, a_{2} ] $ and $\text{Es} (\cdot )$ is as defined in Section \ref{ONE}. 
Note that by \eqref{ES} we have $\text{Es} (r) \asymp r^{- (2- \beta )}$. Thus, we need to give a bound on the Green's function of $Y$. To do it, we set $H^{a} := \big\{ (a, y^{2}, y^{3} ) \in H (a) \ \big| \  ( | y^{2} | - 2^{-m} ) \, (|y^{3}| - 2^{-m} ) = 0 \big\}$ for the union of four edges of the square $H (a)$, and let
\begin{equation}
l_{y} = \text{dist} (y, H^{y^{1}} )
\end{equation}
for $y = (y^{1}, y^{2}, y^{3} ) \in H [a_{1} -q, a_{2} ]$. We first consider the following two cases for the location of $y$.

\vspace{1mm}

$\bullet$ \underline{{\bf Case 1:}} $ a_{1} \le y^{1} \le a_{1} + 2^{-m}$ and $0 \le l_{y} \le 2^{-m -3}$.

\vspace{1mm}

\hspace{-6.5mm} In this case, using \eqref{srwbound}, \eqref{srwbound-2} and the gambler's ruin estimate (see \eqref{Gambler} for this) we have 
\begin{equation*}
G_{H [a_{1} -q, a_{2} ]} \big( x, y; Y \big) \le C 2^{2 m} 2^{-n} l_{y} \ \ \text{ and }  \  \ \text{Es} \big( 2^{n} d_{y} \big) 
= \text{Es} \big( 2^{n} l_{y} \big) \le C \big( 2^{n} l_{y} \big)^{- ( 2 -\beta )}.
\end{equation*}

\vspace{1mm}

$\bullet$ \underline{{\bf Case 2:}} $ a_{1} \le y^{1} \le a_{1} + 2^{-m}$ and $2^{-m -3} \le  l_{y} \le 2^{-m }$.

\vspace{1mm}

\hspace{-6.3mm} In this case, using \eqref{srwbound}, \eqref{srwbound-2} and the gambler's ruin estimate again we have 
\begin{equation*}
G_{H [a_{1} -q, a_{2} ]} \big( x, y; Y \big) \le C \big( y^{1} + 2^{-m} - l_{y} \big) 2^{-n}  \ \ \text{ and }   \ \  \text{Es} \big( 2^{n} d_{y} \big) \le C \big\{  \big( y^{1} + 2^{-m} - l_{y} \big) 2^{n} \big\}^{- (2 - \beta ) }.
\end{equation*}

Combining these estimates, it holds that 
\begin{equation*}
\sum_{y \in H \big[ a_{1}, a_{1} + 2^{-m} \big]} P ( y \in \iota ) \le C 2^{- \beta m} 2^{\beta n}.
\end{equation*}
A similar estimate shows that 
\begin{equation*}
\sum_{y \in H \big[ a_{1} - q, a_{1} \big]} P ( y \in \iota ) \le C 2^{- \beta m} 2^{\beta n} \ \ \text{ and }  \ \  \sum_{y \in H \big[a_{1} + 2^{-m}, a_{2} \big]} P ( y \in \iota ) \le C 2^{- \beta m} 2^{\beta n}.
\end{equation*}
Thus, we conclude that 
\begin{equation*}
E^{x}  \big( {\rm len} (\iota) \big) \le C 2^{- \beta m} 2^{\beta n} \ \ \text{ uniformly in } x \in G (a_{1} )
\end{equation*}
for some universal constant $C < \infty$.
The inequality \eqref{0902-f-1} then follows from Markov's inequality. This finishes the proof of the lemma.
\end{proof}

Recall that $T_{r} := T^{S}_{r} = \inf \{ j \ge 0 \ | \  |S (j) | \ge r \}$ stands for the first exit time from $B (r)$.  Write 
\begin{equation}\label{0908}
\gamma_{n, m} = \text{LE} \Big( S \big[ 0, T_{40 \cdot m_{0} \,  2^{-m}} \big] \Big).
\end{equation}
Note that $H \big[ a_{0}, a_{2 m_{0}} \big] \subset B \big( 10 \cdot m_{0} \,  2^{-m} \big) $ by our construction.
Take $w = (w^{1}, w^{2}, w^{3}) \in   B \big( 10 \cdot m_{0} \,  2^{-m} \big) $. We write 
\begin{equation}\label{hw}
H_{w} = \big\{ y = (y^{1}, y^{2}, y^{3} ) \ \big| \ |y^{1} - w^{1} | \le q,  \ |y^{j} - w^{j} | \le 2^{-m} \ \text{ for } j=2,3 \big\}
\end{equation}
for a ``thin'' cuboid centered at $w$, and set
\begin{equation*}
N_{w} = \Big| H_{w} \cap \gamma_{n, m} \Big|,
\end{equation*}
where we recall that  $|A|$ stands for the cardinality of $A \cap 2^{-n} \mathbb{Z}^{3}$. The next lemma gives an upper bound on $N_{w}$ uniformly in $w \in  B \big( 10  m_{0}  \cdot  2^{-m} \big)$ with $\text{dist} (0, H_{w} ) \ge 2^{-m -1}$.

\begin{lem}\label{lem3-critical}
There exist universal constants $0 < c, C < \infty$ such that 
\begin{equation}\label{0903-a-1'}
P \big( N_{w} \ge 2^{-\beta m} 2^{\beta n} \big) \le C \exp \big\{ - c  \, 2^{  ( \beta - 1) m_{0}} \big\} 
\end{equation}
uniformly in $ w \in  B \big( 10 \cdot m_{0} \,  2^{-m} \big)$  with $\text{dist} \, (0, H_{w} ) \ge 2^{-m -1}$.
\end{lem}

\begin{proof}

Take $w \in  B \big( 10  m_{0}  \cdot  2^{-m} \big)$ with $\text{dist} (0, H_{w} ) \ge 2^{-m -1}$. For simplicity, we write $B := B ( 40 \cdot m_{0} 2^{-m} )$ throughout the proof of the lemma.  We apply Proposition \ref{prop:npoint1} to show that for $x \in H_{w}$
\begin{equation*}
P (x \in \gamma_{n,m} ) \le C G_{B } (0, x; S ) \, \text{Es} ( 2^{n} |x| ),
\end{equation*} 
where $G_{A} (z, w ; S)$ stands for the Green's function of $S$ in $A$ and $\text{Es} (\cdot )$ is as defined in Section \ref{ONE}. 
Since $2^{-m-1} \le |x| \le 20 \cdot m_{0} 2^{-m}$, we can use \eqref{srwbound-2} to obtain that $  G_{B } (0, x; S ) \le C (2^{n} |x|)^{-1} \le C 2^{m} 2^{-n} $. Furthermore, using $\text{Es} (r) \asymp r^{- (2- \beta)}$ (see  \eqref{ES} for this), it follows that 
the probability of $x$ lying on $\gamma_{n,m}$ is bounded above by $C \big( 2^{n} 2^{-m} \big)^{-3 + \beta}$. Note that $|H_{w} | \le C 2^{-m_{0} } 2^{- 3 m} 2^{3 n} $. Thus, we have 
\begin{align*}
 &E (N_{w}) = \sum_{x \in H_{w}} P (x \in \gamma_{n,m} ) \\
 & \le C 2^{-m_{0}} 2^{-\beta m} 2^{\beta n}  \text{ uniformly in $w \in  B \big( 10 \cdot m_{0} \,  2^{-m} \big)$ with $\text{dist} \, (0, H_{w} )   \ge 2^{-m -1}$,}
 \end{align*}
for some universal constant $C \in (0, \infty)$.

We also need a similar bound on $E (N_{0} )$. To make it, we consider a collection of cubes $\{ D_{j} \}_{j=0}^{r}$ of side length $2q$  for which the following conditions are satisfied.

\begin{itemize}
\item $D_{j} = \{ x \in \mathbb{R}^{3} \ | \ \| x- x_{j} \|_{\infty} \le q \}$ for $0 \le j \le r$ and $x_{0} =0$.

\item $\| x_{i} - x_{j} \|_{\infty} \ge \frac{3q }{2} $ if $i \neq j$.

\item $H_{0} \subset \bigcup_{j=0}^{r} D_{j}$ and $r \asymp \Big( \frac{2^{-m}}{q} \Big)^{2} = 2^{2 m_{0}}$.

\end{itemize}
Using Proposition \ref{prop:npoint1} again, we have that  
\begin{equation*}
E \Big( \big| D_{0} \cap \gamma_{n,m} \big| \Big) = \sum_{x \in D_{0} } P ( x \in \gamma_{n,m} ) \le C \sum_{x \in D_{0} } G_{B} (0, x; S) \text{Es} ( 2^{n} |x| ) \asymp \sum_{l = 1}^{q 2^{n}} l^{-1 + \beta} \asymp (q 2^{n} )^{\beta},
\end{equation*}
where we used the fact that $\text{Es} (r) \asymp r^{- (2- \beta)}$ (see  \eqref{ES} for this) in the third equation. 

Take $j \neq 0$. Note that $\| x_{j} \|_{\infty} \ge \frac{3 q}{2}$. Thus, we can choose an integer $k \ge 1$ such that $kq \le \| x_{j} \|_{\infty} < (k + 1) q $. It then follows that 
\begin{equation*}
E \Big( \big| D_{j} \cap \gamma_{n,m} \big| \Big) = \sum_{x \in D_{j} } P ( x \in \gamma_{n,m} ) \le C \sum_{x \in D_{j} } G_{B} (0, x; S) \text{Es} ( 2^{n} |x| ) \asymp k^{-3 + \beta} (q 2^{n} )^{\beta},
\end{equation*}
where we used the fact that $G_{B} (0, x; S) \text{Es} ( 2^{n} |x| ) \asymp (k q 2^{n} )^{-3 + \beta}$ for $x \in D_{j}$ in the last equation. Since the number of $j \in \{ 0, 1, \cdots , r \}$ satisfying $kq \le \| x_{j} \|_{\infty} < (k + 1) q $ is comparable to $k$, 
we conclude that 
\begin{equation}\label{osaka}
E (N_{0} ) \le \sum_{j=0}^{r} E \Big( \big| D_{j} \cap \gamma_{n,m} \big| \Big) \le C (q 2^{n} )^{\beta} + C \sum_{k=1}^{2^{m_{0}}} k \cdot  
k^{-3 + \beta} (q 2^{n} )^{\beta} \le C 2^{-m_{0}} 2^{- \beta m} 2^{\beta n}.
\end{equation}

Take an integer $\vartheta \ge 2$. We will next give an upper bound on $E ( N_{w}^{\vartheta} )$ for $w \in  B \big( 10  m_{0}  \cdot  2^{-m} \big)$ with $\text{dist} (0, H_{w} ) \ge 2^{-m -1}$.  We use Proposition \ref{prop:npoint1} (see also the sixth line of the proof of \cite[Theorem 8.4]{S}) to show that
\begin{equation*}
E (N_{w}^{\vartheta} ) \le C^{\vartheta} \vartheta! \sum_{z_{1} \in H_{w}} \sum_{z_{2} \in H_{w}} \cdots \sum_{z_{\vartheta} \in H_{w}} \prod_{i=1}^{\vartheta} G_{B} (z_{i-1}, z_{i}; S) \text{Es} ( 2^{n} d_{i} ) 
\end{equation*}
for some universal constant $C \in (0, \infty)$,  where we set $z_{0} = 0$ and define $d_{i}$ by 
\begin{equation*}
d_{i} = |z_{i} -z_{i-1} | \wedge |z_{i} - z_{i+1} | \text{ \  for \  $1 \le i \le \vartheta-1$ \ \ \  and \ \ \  $d_{\vartheta} = |z_{\vartheta} - z_{\vartheta-1}|$.}
\end{equation*}

Write 
\begin{equation*}
L_{\vartheta} = \sum_{z_{1} \in H_{w}} \sum_{z_{2} \in H_{w}} \cdots \sum_{z_{\vartheta} \in H_{w}} \prod_{i=1}^{\vartheta} G_{B} (z_{i-1}, z_{i}; S) \text{Es} ( 2^{n} d_{i} ).
\end{equation*}
We will first deal with the terms containing $z_{\vartheta}$. Since $$\text{Es} (2^{n} d_{\vartheta-1} ) \le \text{Es} (2^{n} | z_{\vartheta-1} - z_{\vartheta-2} | ) + \text{Es} ( 2^{n} |z_{\vartheta-1} - z_{\vartheta} | ),$$ we can decompose $L_{\vartheta}$ as follows. 
\begin{align}\label{lq}
L_{\vartheta} &=  \sum_{z_{1} \in H_{w}} \sum_{z_{2} \in H_{w}} \cdots \sum_{z_{\vartheta-1} \in H_{w}} \prod_{i=1}^{\vartheta-2} G_{B} (z_{i-1}, z_{i}; S) \text{Es} ( 2^{n} d_{i} ) \cdot G_{B} (z_{\vartheta-2}, z_{\vartheta-1}; S) \notag  \\
&  \ \ \ \ \ \ \ \ \ \ \ \ \ \ \ \ \ \ \ \ \ \ \ \ \ \ \ \ \ \  \ \ \ \ \ \ \ \ \ \   \times \sum_{z_{\vartheta} \in H_{w} } G_{B} (z_{\vartheta-1}, z_{\vartheta}; S) \text{Es} (2^{n} d_{\vartheta} ) \text{Es} (2^{n} d_{\vartheta-1} ) \notag \\
&\le  \sum_{z_{1} \in H_{w}} \sum_{z_{2} \in H_{w}} \cdots \sum_{z_{\vartheta-1} \in H_{w}} \prod_{i=1}^{\vartheta-2} G_{B} (z_{i-1}, z_{i}; S) \text{Es} ( 2^{n} d_{i} ) \cdot G_{B} (z_{\vartheta-2}, z_{\vartheta-1}; S) \notag  \\
&  \ \ \ \ \ \ \ \ \  \ \ \ \   \times \sum_{z_{\vartheta} \in H_{w} } G_{B} (z_{\vartheta-1}, z_{\vartheta}; S) \text{Es} (2^{n} d_{\vartheta} ) \big\{ \text{Es} (2^{n} | z_{\vartheta-1} - z_{\vartheta-2} | ) + \text{Es} ( 2^{n} |z_{\vartheta-1} - z_{\vartheta} | ) \big\} \notag \\
&= \sum_{z_{1} \in H_{w}} \sum_{z_{2} \in H_{w}} \cdots \sum_{z_{\vartheta-1} \in H_{w}} \prod_{i=1}^{\vartheta-2} G_{B} (z_{i-1}, z_{i}; S) \text{Es} ( 2^{n} d_{i} ) \cdot G_{B} (z_{\vartheta-2}, z_{\vartheta-1}; S)   \notag   \\
&  \ \ \ \ \ \ \ \ \ \ \ \ \  \ \ \ \   \times \text{Es} (2^{n} | z_{\vartheta-1} - z_{\vartheta-2} | )\sum_{z_{\vartheta} \in H_{w} } G_{B} (z_{\vartheta-1}, z_{\vartheta}; S) \text{Es} (2^{n} d_{\vartheta} )  \notag   \\
&+ \sum_{z_{1} \in H_{w}} \sum_{z_{2} \in H_{w}} \cdots \sum_{z_{\vartheta-1} \in H_{w}} \prod_{i=1}^{\vartheta-2} G_{B} (z_{i-1}, z_{i}; S) \text{Es} ( 2^{n} d_{i} ) \cdot G_{B} (z_{\vartheta-2}, z_{\vartheta-1}; S) \notag  \\
&  \ \ \ \ \ \ \ \ \ \ \ \ \ \ \ \ \ \ \ \ \ \ \ \ \ \ \ \ \ \  \ \ \ \   \times \sum_{z_{\vartheta} \in H_{w} } G_{B} (z_{\vartheta-1}, z_{\vartheta}; S) \text{Es} (2^{n} d_{\vartheta} )^{2}  =: L_{\vartheta, 1} + L_{\vartheta,2}.
\end{align}

For each $z= (z^{1}, z^{2}, z^{3})$, we write 
\begin{equation*}
H'_{z} = \big\{ y = (y^{1}, y^{2}, y^{3} ) \ \big| \ |y^{1} - w^{1} | \le 2q,  \ |y^{j} - w^{j} | \le 2^{-m+1} \ \text{ for } j=2,3 \big\},
\end{equation*}
which is an enlargement of $H_{z}$, see \eqref{hw}.
Notice that if $z_{\vartheta-1} \in H_{w}$ then $H_{w} \subset H_{z_{\vartheta-1}}'$. Thus, by the translation invariance, a similar calculation as in \eqref{osaka} shows that 
\begin{align*}
& \sum_{z_{\vartheta} \in H_{w} } G_{B} (z_{\vartheta-1}, z_{\vartheta}; S) \text{Es} (2^{n} d_{\vartheta} )  \le \sum_{z_{\vartheta} \in H_{z_{\vartheta-1}}'} G_{B} (z_{\vartheta-1}, z_{\vartheta}; S) \text{Es} (2^{n} d_{\vartheta} ) \\
 &= \sum_{z \in H_{0}'} G_{B} (0, z; S) \text{Es} (2^{n} |z| ) \le C 2^{-m_{0}} 2^{- \beta m} 2^{\beta n}.
 \end{align*}
This implies that $L_{\vartheta,1}$ defined as in \eqref{lq} is bounded above by
\begin{align}\label{osaka-3}
 C 2^{-m_{0} - \beta m + \beta n} \sum_{z_{1} \in H_{w}}  &\cdots \sum_{z_{\vartheta-1} \in H_{w}}\\
 & \prod_{i=1}^{\vartheta-2} G_{B} (z_{i-1}, z_{i}; S) \text{Es} ( 2^{n} d_{i} ) \,  G_{B} (z_{\vartheta-2}, z_{\vartheta-1}; S)  \text{Es} (2^{n} | z_{\vartheta-1} - z_{\vartheta-2} | ).
\end{align}

We will consider $L_{\vartheta,2}$. For each $z_{\vartheta-1} \in H_{w}$, we have 
\begin{align*}
 \sum_{z_{\vartheta} \in H_{w} } G_{B} (z_{\vartheta-1}, z_{\vartheta}; S) \text{Es} (2^{n} d_{\vartheta} )^{2} \le  \sum_{z_{\vartheta} \in H_{z_{\vartheta-1}}' } G_{B} (z_{\vartheta-1}, z_{\vartheta}; S) \text{Es} (2^{n} d_{\vartheta} )^{2}.
 \end{align*}
Using the translation invariance again, a similar argument used in \eqref{osaka} gives that
\begin{align}\label{osaka-2}
&\sum_{z_{\vartheta} \in H_{z_{\vartheta-1}}' } G_{B} (z_{\vartheta-1}, z_{\vartheta}; S) \text{Es} (2^{n} d_{\vartheta} )^{2} = \sum_{z \in H_{0}' } G_{B} (0, z; S) \text{Es} (2^{n} |z| )^{2} \notag \\
&\le \sum_{\| z \|_{\infty} \le 2q } G_{B} (0, z; S) \text{Es} (2^{n} |z| )^{2} + \sum_{z \in H_{0}', \, \| z \|_{\infty} \ge 2q } G_{B} (0, z; S) \text{Es} (2^{n} |z| )^{2} \notag  \\
&\le C \sum_{l=1}^{q 2^{n}} l^{-3 + 2 \beta} + C (q 2^{n} )^{-2 + 2 \beta} \sum_{k=1}^{2^{m_{0}}} k^{-4 + 2 \beta}.
\end{align}
Since $\beta > 1$, we see that the first term of \eqref{osaka-2} is bounded above by $C (q 2^{n} )^{-2 + 2 \beta}$. We need to be careful with the second term of \eqref{osaka-2} since we cannot exclude the case that $-4 + 2 \beta = -1$ (which is the case that $\beta = 3/2$). If $-4 + 2 \beta > -1$ the sum in the second term of \eqref{osaka-2} is smaller than $C (2^{m_{0}})^{-3 + 2 \beta } $, while it is bounded above by $C m_{0}$ if $-4 + 2 \beta \le -1$. Thus, we have 
\begin{equation*}
\text{\eqref{osaka-2}} \le C m_{0} (q 2^{n} )^{-2 + 2 \beta} + C ( q 2^{n} )^{-2 + 2 \beta} (2^{m_{0}})^{-3 + 2 \beta + \epsilon}.
\end{equation*}

We note that $(q 2^{n} )^{-2 + 2 \beta}$ is bounded above by 
\begin{equation*}
C 2^{- 2 ( \beta - 1) m_{0} } 2^{-\beta m} 2^{\beta n} \text{Es} ( 2^{n} 2^{-m} ) \le C 2^{- 2 ( \beta - 1) m_{0} } 2^{-\beta m} 2^{\beta n} \text{Es} ( 2^{n} | z_{\vartheta-1} - z_{\vartheta-2} |)
\end{equation*}
whenever $z_{\vartheta-2}, z_{\vartheta-1} \in H_{w}$ since $| z_{\vartheta-1} - z_{\vartheta-2} | \le 2^{-m +3}$. Furthermore, the quantity $(q 2^{n} )^{-2 + 2 \beta}  (2^{m_{0}})^{-3 + 2 \beta }$ is less than 
\begin{equation*}
C (2^{m_{0}})^{-3 + 2 \beta } 2^{- 2 ( \beta - 1) m_{0} } 2^{-\beta m} 2^{\beta n} \text{Es} ( 2^{n} 2^{-m} ) \le C 2^{-  m_{0} } 2^{-\beta m} 2^{\beta n} \text{Es} ( 2^{n} | z_{\vartheta-1} - z_{\vartheta-2} |)
\end{equation*}
uniformly in $z_{\vartheta-2}, z_{\vartheta-1} \in H_{w}$. Combined  with \eqref{osaka-3},This implies that 
\begin{align*}
 L_{\vartheta}=L_{\vartheta,1}+L_{\vartheta,2}<\; &C 2^{-  ( \beta - 1) m_{0}} 2^{- \beta m} 2^{\beta n}  \sum_{z_{1} \in H_{w}} \sum_{z_{2} \in H_{w}} \cdots \sum_{z_{\vartheta-1} \in H_{w}} \\ 
 &\quad\prod_{i=1}^{\vartheta-2} G_{B} (z_{i-1}, z_{i}; S) \text{Es} ( 2^{n} d_{i} ) \,  G_{B} (z_{\vartheta-2}, z_{\vartheta-1}; S)  \text{Es} (2^{n} | z_{\vartheta-1} - z_{\vartheta-2} | ),
\end{align*}
where we used the fact that $m_{0} 2^{- 2 ( \beta - 1) m_{0} } + 2^{-m_{0}} \le C 2^{-  ( \beta - 1) m_{0}}$ since $\beta \in (1, \frac{5}{3}]$.
Iterating this $\vartheta-1$ times, it follows that 
\begin{equation*}
E (N_{w}^{\vartheta} ) \le C^{\vartheta} \vartheta!  \big\{ 2^{-  ( \beta - 1) m_{0}} 2^{- \beta m} 2^{\beta n} \big\}^{\vartheta},
\end{equation*}
which implies that 
\begin{equation*}
E \bigg[ \exp \bigg\{ \frac{ c  N_{w} }{ 2^{-  (\beta -1 ) m_{0}} 2^{-\beta m} 2^{\beta n} } \bigg\} \bigg]  \le C
\end{equation*}
uniformly in $w \in  B \big( 10 \cdot m_{0} \,  2^{-m} \big)$  with $\text{dist} \, (0, H_{w} ) \ge 2^{-m -1}$, for some universal constant $0 < c, C < \infty$. By Markov's inequality, this implies \eqref{0903-a-1'}.
\end{proof}

Set 
\begin{equation}\label{b}
J_{i} = \bigcap_{k=0}^{i} D_{k}  (C_{\ast} ) 
\end{equation}
where the event $D_{k} (C)$  and the constant $C_{\ast}$ are as defined in \eqref{c} and \eqref{0902-d-2} respectively. We also define for $i\ge1$
\begin{align}\label{ab}
&w_{i} = ( a_{i}, 0, 0 ) = \big( 2^{-m} (2i-1), 0, 0 \big),  \ \ \  \ \ \ L_{i} = \Big\{ \big| H_{w_{l}} \cap \xi_{i} \big| \le 2^{- \beta m} 2^{\beta n} \ \text{ for all } 1 \le l \le \frac{i}{2} \Big\}  \notag \\
&\text{and} \ \   U_{ 2 m_{0}} = \Big\{ S \big( T_{40 \cdot m_{0} \,  2^{-m}} \big) \in \big\{ (y^{1}, y^{2}, y^{3} ) \in \mathbb{R}^{3} \ \big| \ y^{1} \ge 8 m_{0} \cdot 2^{-m} \big\},   \notag \\ 
& \  \ \ \ \  \ \ \  \ \    \ \  \ \ \  \ \ \  \ \ S \big[ t_{S} (a_{ 2 m_{0}+1} ) , T_{40 \cdot m_{0} \,  2^{-m}} \big] \cap B \big( a_{\frac{3}{2} \cdot  m_{0}}  \big)  = \emptyset \Big\},
\end{align}
see \eqref{hw} and \eqref{09070} for $H_{w}$ and $\xi_{i}$ (we recall that $|A|$ stands for the cardinality of $A \cap 2^{-n} \mathbb{Z}^{3}$).

We write
\begin{equation}\label{am}
A^{m} = F_{2 m_{0}} \cap G_{2 m_{0}} \cap J_{2 m_{0}} \cap U_{2m_{0}} \cap L_{2 m_{0}}
\end{equation}
for the intersection of all events which have been defined. Here we recall that $F_{i}$, $G_{i}$, $J_{i}$, $U_{2 m_{0}}$ and $L_{i}$ are as defined in \eqref{0907-b-1}, \eqref{0907-b-2}, \eqref{b} and \eqref{ab} respectively.

The next proposition then gives a lower bound of the probability of this event $A^m$.

\begin{prop}\label{lem-critical}
There exists a universal constant $c > 0$ such that 
\begin{equation}\label{o''}
P \Big( A^{m}  \Big) \ge c 2^{- 3 m_{0}^{2}},
\end{equation}
where $A^{m}$ is defined as in \eqref{am}.
\end{prop}

\begin{proof}
 Note that by the strong Markov property and \eqref{srwbound}, there exists a universal constant $c > 0$ such that 
\begin{equation}\label{0903}
P \big( U_{2 m_{0}} \ \big| \  F_{2 m_{0}} \cap G_{2 m_{0}} \cap J_{2 m_{0}} \big) \ge c.
\end{equation}
Furthermore, using \eqref{0902-d-2} and  the strong Markov property again, it holds that 
\begin{equation*}
P \big( F_{2 m_{0}} \cap G_{2 m_{0}} \cap J_{2 m_{0}} \big) \ge c_{\ast}^{2 m_{0}} 2^{- 2 m_{0}^{2}} \ge 2^{- 3 m_{0}^{2}},
\end{equation*}
where if necessary we change  $m_{0}$ so that the last inequality above holds. Combining this with \eqref{0903}, we have 
\begin{equation}\label{0903-b-1}
P \big( F_{2 m_{0}} \cap G_{2 m_{0}} \cap J_{2 m_{0}} \cap U_{2m_{0}} \big) \ge c 2^{- 3 m_{0}^{2}}.
\end{equation}

Suppose that the event $F_{2 m_{0}} \cap G_{2 m_{0}} \cap J_{2 m_{0}} \cap U_{2m_{0}} $ occurs. Since $$S \big[ t_{S} (a_{2 m_{0}+1} ) ,  T_{40 \cdot m_{0} \,  2^{-m}}  \big] \mbox{ does not hit }B \big( a_{\frac{3}{2} \cdot  m_{0}}  \big)$$ and $\xi_{2m_{0}}$ has no big backtracking in the sense that $$\xi_{2 m_{0}} \big[  t_{ \xi_{2 m_{0}} } \big( a_{\frac{3}{2} \cdot  m_{0}} \big), {\rm len} \big( \xi_{2 m_{0}} \big) \big]\mbox{ does not hit  } B \big( a_{m_{0} + 1} \big),$$  we see that 
\begin{align}\label{a}
&\xi_{2 m_{0}} \big[ 0, t_{ \xi_{2 m_{0}} } \big( a_{\frac{3}{2} \cdot  m_{0}} \big) \big] = \gamma_{n, m} \big[ 0, t_{ \gamma_{n,m} } \big(  
a_{\frac{3}{2} \cdot  m_{0}} \big) \big] \notag \\
& \text{and thus }  \   \ H_{w_{l}} \cap \xi_{2 m_{0}} = H_{w_{l}} \cap \gamma_{n, m} \ \  \text{ for all }  \ 1 \le l \le m_{0}.
\end{align}
Therefore, if the event $L_{2 m_{0}}$ does not occur, there exists $1 \le l \le m_{0}$ such that $\big| H_{w_{l}} \cap \gamma_{n, m} \big|$ is bounded below by $2^{- \beta m} 2^{\beta n}$. Since $w_{l} \in B ( 10 \cdot m_{0} 2^{-m} )$ and  $\text{dist} (0, H_{w_{l}} ) \ge 2^{-m-1}$, we may use \eqref{0903-a-1'} to show that 
\begin{align*}
&P \Big( F_{2 m_{0}} \cap G_{2 m_{0}} \cap J_{2 m_{0}} \cap U_{2m_{0}} \cap (L_{2 m_{0}} )^{c} \Big) \notag\\
\le\;& P \Big( 1 \le \exists \,  l \le m_{0} \text{ s.t. } \big| H_{w_{l}} \cap \gamma_{n, m} \big| \ge 2^{- \beta m} 2^{\beta n} \Big) \notag \\
 \le \;& C m_{0} \, \exp \big\{ - c  \, 2^{ (\beta -1) m_{0}} \big\}.
\end{align*}
Combining this with \eqref{0903-b-1} (and the fact that $\beta > 1$), we conclude that 
\begin{equation}\label{o'}
P \Big( F_{2 m_{0}} \cap G_{2 m_{0}} \cap J_{2 m_{0}} \cap U_{2m_{0}} \cap L_{2 m_{0}}  \Big) \ge c 2^{- 3 m_{0}^{2}},
\end{equation}
which finishes the proof.
\end{proof}

We will extract a necessary property of $\gamma_{n,m}$ from the event $A^{m}$ in the next lemma.

\begin{lem}\label{lem-2-critical}
On the event $A^{m}$ defined as in \eqref{am}, it holds that 
\begin{equation}\label{q'}
t_{\gamma_{n,m}} \big( a_{m_{0}} \big)  \le \widehat{C}  m_{0} \, 2^{- \beta m} \, 2^{ \beta n} \ \ \text{ and }  \  \ \Big| \gamma_{n,m} \Big( t_{\gamma_{n,m}} \big( a_{m_{0}} \big) \Big) \Big| \ge m_{0} 2^{- m},
\end{equation}
for some universal constant $\widehat{C} < \infty$. Here recall that $\gamma_{n,m}$ and $t_{\lambda} (a)$ were given as in \eqref{0908} and \eqref{t-lambda}.
\end{lem}

\begin{proof}
It is routine to show that on the event  $A^{m}$ the following six assertions hold:
\begin{itemize}
\item[(1)] $\xi_{2 m_{0}} \big[ 0, t_{ \xi_{2 m_{0}} } \big( a_{\frac{3}{2} \cdot  m_{0}} \big) \big] = \gamma_{n, m} \big[ 0, t_{ \gamma_{n,m} } \big(  
a_{\frac{3}{2} \cdot  m_{0}} \big) \big]$ \  by \eqref{a}.

\item[(2)] $\xi_{m_{0}} \big[0, t_{\xi_{m_{0}}} \big( a_{m_{0} + 1} -q \big) \big] =  \xi_{2 m_{0} } \big[0, t_{\xi_{2 m_{0}}} \big( a_{m_{0} + 1} -q \big) \big] $  since $S \big[ t_{S} \big( a_{m_{0} + 1} \big),  t_{S} \big( a_{2 m_{0}} \big) \big]$ does not hit $H (a_{m_{0}+ 1} -q ) $ by the event $F_{2 m_{0}}$.

\item[(3)] $H_{w_{l}} \cap \, \xi_{2 m_{0}} = H_{w_{l}} \cap \, \xi_{m_{0} } \ \  \text{ for all }  \ 1 \le l \le m_{0}$ thanks to (2) and the fact that $S \big[ t_{S} ( a_{m_{0} + 1} -q ), t_{S} (a_{2 m_{0}} ) \big] $ does not hit $H_{w_{m_{0}}}$.

\item[(4)] ${\rm len} (\xi_{m_{0} } ) \le  {\rm len} (\xi_{0} ) + \sum_{l = 1}^{m_{0} } \Big\{ {\rm len} ( \lambda_{l} ) + \Big| H_{w_{l}} \cap \, \xi_{2 m_{0}} \Big| \Big\}$ \  by \eqref{0902'} and (3)
(notice that $H_{w_{l}} = H \big[ a_{l} - q, a_{l} + q \big]$).

\item[(5)] ${\rm len} ( \xi_{0} ) \le C_{\ast} 2^{- \beta m} 2^{\beta n}$ \  and \  ${\rm len} (\lambda_{i} )  \le C_{\ast} 2^{- \beta m } 2^{\beta n}$ for all $1 \le i \le 2m_{0}$ \  by \eqref{c} and \eqref{b}.

\item[(6)] $\Big| H_{w_{l}} \cap \, \xi_{2 m_{0}} \Big| \le 2^{- \beta m} 2^{\beta n}$ for all $1 \le l \le m_{0}$ by the event $L_{2 m_{0}} $.

\end{itemize}

Combining these facts, it holds that 
\begin{equation}\label{q}
t_{\gamma_{n,m}} \big( a_{m_{0}} \big) = t_{ \xi_{2 m_{0}} } \big( a_{ m_{0}} \big) = t_{\xi_{m_{0}}} \big( a_{m_{0}}  \big) \le {\rm len} (\xi_{m_{0}} ) \le \widehat{C}  m_{0} \, 2^{- \beta m} \, 2^{ \beta n}
\end{equation}
on the event $A^{m}$. Note that by definition $ \Big| \gamma_{n,m} \Big( t_{\gamma_{n,m}} \big( a_{m_{0}} \big) \Big) \Big| \ge m_{0} 2^{- m}$. So the proof is completed.
\end{proof}

Finally, we reach our goal of this subsection in the next proposition. Write 
\begin{equation}\label{0908-1}
\eta_{n,m} (t ) = \gamma_{n,m} \big(  2^{\beta n} t \big) \ \ \text{ for } \ 0 \le t \le 2^{- \beta n} \, {\rm len} \big( \gamma_{n,m} \big)
\end{equation}
for the time-rescaled version of $\gamma_{n,m}$ (recall that $\gamma_{n,m}$ is defined as in \eqref{0908})

\begin{prop}\label{cor-cri-1}
There exist a universal constant $c > 0$ a such that on the event $A^{m}$ defined as in \eqref{am} 
\begin{equation}\label{o}
\sup_{0 \le s < t \le t_{\eta_{n,m}}} \frac{|\eta_{n,m} (s) - \eta_{n,m} (t) |}{|s-t|^{\frac{1}{\beta}}} \ge c \, m_{0}^{1 - \frac{1}{\beta}},
\end{equation}
where $\eta_{n,m}$ is defined as in \eqref{0908-1}.
\end{prop}
\begin{proof}
By \eqref{q}, 
\begin{equation}\label{g}
\sup_{0 \le s < t \le t_{\eta_{n,m}}} \frac{|\eta_{n,m} (s) - \eta_{n,m} (t) |}{|s-t|^{\frac{1}{\beta}}} \ge \frac{ \Big| \gamma_{n,m} \Big( t_{\gamma_{n,m}} \big( a_{m_{0}} \big) \Big) \Big| }{ \Big\{ 2^{- \beta n} \,  t_{\gamma_{n,m}} \big( a_{m_{0}} \big) \Big\}^{\frac{1}{\beta}}  } \ge c \, m_{0}^{1 - \frac{1}{\beta}}
\end{equation}
on the event $A^{m}$. Here we considered the case that $s=0$ and $t= 2^{- \beta n} \, t_{\gamma_{n,m}} \big( a_{m_{0}} \big)$ in the first inequality of \eqref{g}. 
\end{proof}

\subsection{The iteration argument}
In the previous subsection, we have shown that LERW is not $1/\beta$-H\"older continuous under the event $A^{m}$ defined in \eqref{am} which occurs with probability bounded from below by a constant.  
We want to boost this result by finding an event $A'$  where we have a good control of modulus of continuity as in \eqref{o} but with probability close $1$. For that purpose,  we will make use of iteration arguments in a multi-scale analysis similar to those in \cite[Proposition 6.6]{BM} and \cite[Proposition 8.11]{S}.

We start with our setup. Recall the definition of $m$ and $m_0$ from \eqref{eq:mdef}. Note that we have assumed $2^{-n}$ is smaller than $\exp \big\{ - 2^{m^{100}} \big\}$.

\begin{figure}[h]
\begin{center}
\includegraphics[scale=0.65]{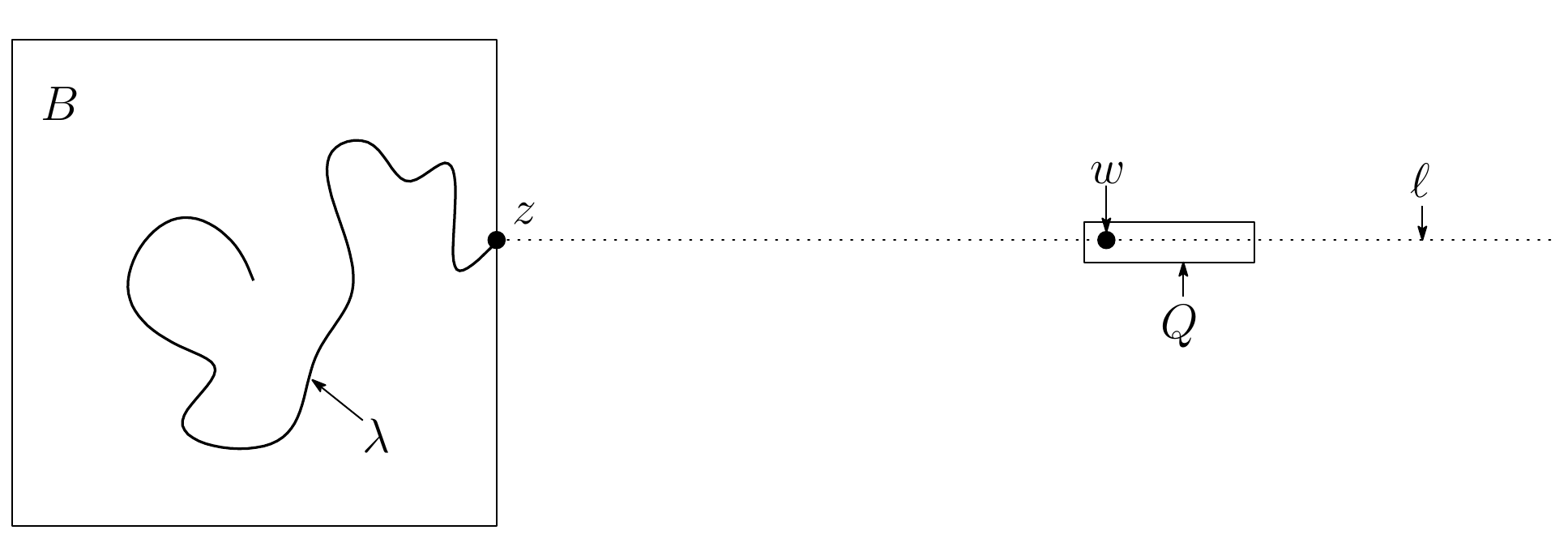}
\caption{Setup in Definition \ref{DEF-it}.}\label{TR-f}
\end{center}
\end{figure}

\begin{dfn}\label{DEF-it}
Take an integer $M > 0$ and write $B = \{ x \in \mathbb{R}^{3} \ | \  \| x \|_{\infty} \le M \cdot 2^{-m}  \}$ for the cube of side length $ M \cdot 2^{-m + 1}$ centered at the origin. We assume that $B \subset \frac{3}{4} \cdot \mathbb{D}$. 

Consider a simple path $\lambda \subset 2^{-n} \mathbb{Z}^{3}$ such that $\lambda (0) = 0$, $\lambda [0, {\rm len} (\lambda ) - 1] \subset B$ and $\lambda \big( {\rm len} ( \lambda ) \big) \in \partial B$. Set $z = \lambda \big(  {\rm len} ( \lambda ) \big)$. 
By symmetry, without loss of generality, we may assume that $z$ lies in the right face of $B$.
We write $\ell$ for the half line starting from $z$ with $\ell \cap B = \{ z \}$ which is perpendicular to a face of $B$ containing the point $z$ (by our assumption, $\ell$ is parallel to $x$-axis).
Set $w \in \ell$ for the point in the half line such that $|z - w| = 100 m_{0} \cdot  2^{-m}$. Let 
\begin{equation}\label{0905}
Q = w + H \big[ a_{0}, a_{2 m_{0}} \big] 
\end{equation}
be a translated version of the tube $H \big[ a_{0}, a_{2 m_{0}} \big]$ (recall that $H [a, b]$ and $a_{i}$ are defined as in Definition \ref{partition} and \eqref{a_i}). See Figure \ref{TR-f} for the setting.

Finally, we write $X$ for the random walk in $2^{-n} \mathbb{Z}^{3}$ started at $z$ which is conditioned that $X [1, T^{X} ] \cap \lambda = \emptyset$, where $T^{X} : = T^{X}_{1}$ stands for the first time that $X$ exits from $\mathbb{D}$.
\end{dfn}

Recall that $\gamma_{n}$ is defined as in \eqref{gammandef}. We remark that by the domain Markov property (see Section \ref{DMPLEW} for this), conditioned on the event $\gamma_{n} [0, U] = \lambda$, the conditional distribution of $\gamma_{n} [U, {\rm len} (\gamma_{n} ) ]$ the remaining part of $\gamma_{n}$ is given by the loop-erasure of $X [0, T^{X}]$.  Here $U:= \tau^{\gamma_{n}}_{B}$ is the first time that $\gamma_{n}$ exits from $B$.

We will prove that modulus of continuity of $\text{LE} \big( X [0, T^{X}] \big)$ restricted in $Q$ satisfies a similar inequality as described in  \eqref{o} with probability at least $c 2^{- 4 m_{0}^{2}}$, where the constant $c > 0$ is universal which does not depend on $M$ and $\lambda$. This uniform estimate enables us to proceed the iteration argument.

To be more precise, we make the following definitions
\begin{align}
&\bullet \tau_{0} = \inf \big\{ j \ge 0 \ \big| \ \| X (j) - z \|_{\infty} \ge (100 m_{0} - 1) \cdot 2^{-m} \big\}. \notag \\
&\bullet \tau_{1} = \inf \big\{ j \ge 0 \ \big| \ X (j) \in w + H (a_{1} ) \big\} \text{ (see Definition \ref{partition} and \eqref{a_i} for $H (a)$ and $a_{i}$)}; \notag \\
&\bullet \tau_{2} = \inf \big\{ j \ge 0 \ \big| \ X (j) \in w + H \big( a_{2 m_{0} + 1} \big) \big\}; \notag \\
&\bullet \tau_{3} = \inf \big\{ j \ge \tau_{1} \ \big| \ X (j) \notin B \big( X ( \tau_{1} ), 40  m_{0} \cdot 2^{-m} \big) \big\}; \notag \\
&\bullet \text{call $k \in [\tau_{1}, \tau_{3}]$ a {\bf good cut time} for $X$ if }  \notag \\
&\text{ \ \   -- (i) $X [\tau_{1}, k ] \cap X [k+1, \tau_{3} ] = \emptyset$,} \notag \\
&\text{ \ \   -- (ii) $X [\tau_{1}, k ] \subset w+ H \big[ 0, a_{2} + q \big] $, $X (k) \in w + H [ a_{2} + q/2, a_{2} + q ]$} \notag \\
&\text{ \ \ \ \   \ \ \ \ \  \  and $X [ k, \tau_{2} ] \subset w+ H \big[ a_{2},  a_{2 m_{0} + 1}  \big] $,} \notag \\
&\text{ \ \   -- (iii) $ X [\tau_{2}, \tau_{3} ] \cap B \Big( z,   103 m_{0} \cdot 2^{-m } \Big) = \emptyset$,}  \notag \\
&\text{ \ \   -- (iv) $  X (\tau_{3}) \in B \big( w, 41 m_{0} \cdot 2^{-m} \big) \cap \Big( z + \big\{ (y^{1}, y^{2}, y^{3} ) \in \mathbb{R}^{3}  \ \big| \ y^{1} \ge 138 m_{0} \cdot 2^{-m}  \big\} \Big) $;} \label{goodcut}\\
&\text{and call $X(k)$ a {\bf good cut point};}  \notag \\
&\bullet \text{ if $X$ has a good cut time in $[\tau_{1}, \tau_{3} ]$ let $k_{\ast}  $ be the smallest good cut time and set} \notag \\
& \ \ \ \ \ \ \ \  \ \xi = \text{LE} \big( X [ k_{\ast}, \tau_{3} ] \big) \ \ \text{ and } \  \   U_{1} = \inf \big\{ j \ge 0 \ \big| \ \xi (j) \in w + H (a_{m_{0}} ) \big\}, \label{xi-u1}
\end{align}
and then define the events $V_{j}$ as follows:
\begin{align}
&V_{1} = \Big\{ X ( \tau_{0} ) \in   w + G ( a_{0} ), \  \text{dist} \big( X [0, \tau_{0} ], \ell \big) \le m_{0} 2^{-m}    \Big\}    \notag \\
&\text{ \ \ \ \ \ \ \ ($G (a)$ and $\ell$ are defined in \eqref{0902-d-1} and Definition \ref{DEF-it} resp.)}; \notag \\
&V_{2} =  \Big\{ X [ \tau_{0}, \tau_{1} ] \subset w + H \big[ a_{0} - 2^{-m}, a_{1} \big], \  X ( \tau_{1} ) \in w + G (a_{1} ) \Big\}; \notag \\
&V_{3} = \Big\{ X \text{ has a good cut time in $ [\tau_{1}, \tau_{3} ]$} \Big\}; \label{eq:Videf} \\
&V_{4} = \Big\{   U_{1} \le \widehat{C} m_{0} 2^{- \beta m} 2^{\beta n} \Big\} \ \ \text{(the constant $\widehat{C}$ is defined as in Lemma \ref{lem-2-critical})}; \notag \\
&V_{5} = \Big\{ X [\tau_{3}, T^{X} ] \cap B \big( z,   110 m_{0} \cdot 2^{-m } \Big) = \emptyset \big\}.\notag
\end{align}
See Figures \ref{ponta1} and \ref{ponta2} for sketches of the event $V_{i}$, $i=1,2,3$. (All constants above are chosen so that Lemma \ref{lem-con-1} and Proposition \ref{lem-con-2} below hold.)

\begin{figure}[h]
\begin{center}
\includegraphics[scale=0.7]{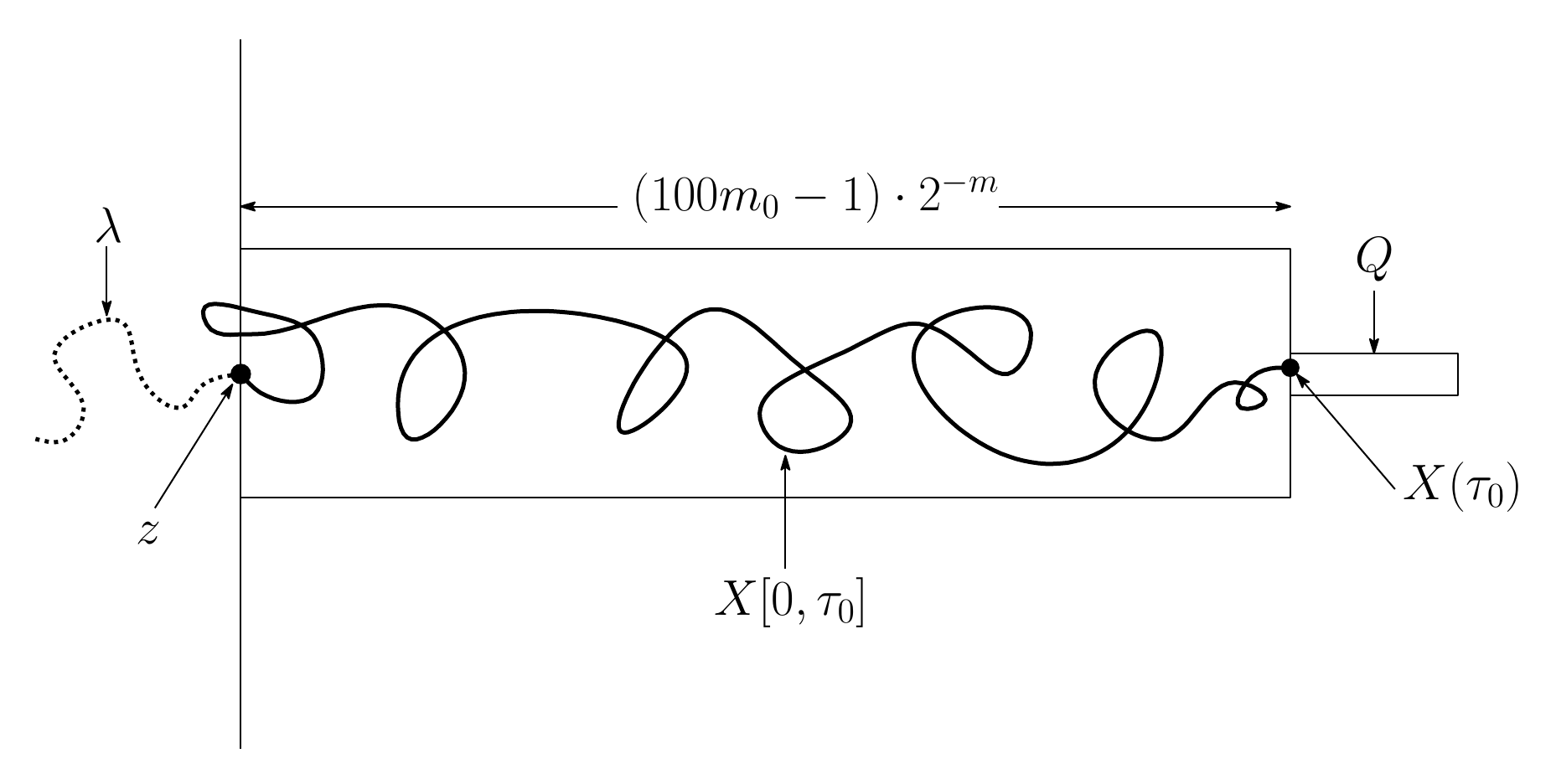}
\caption{Illustration for the event $V_{1}$. $X ( \tau_{0} )$ is located at the left face of $Q$. The height of the long (left) tube is $2m_{0} \cdot 2^{-m}$.}\label{ponta1}
\end{center}
\end{figure}

\begin{figure}[h]
\begin{center}
\includegraphics[scale=0.65]{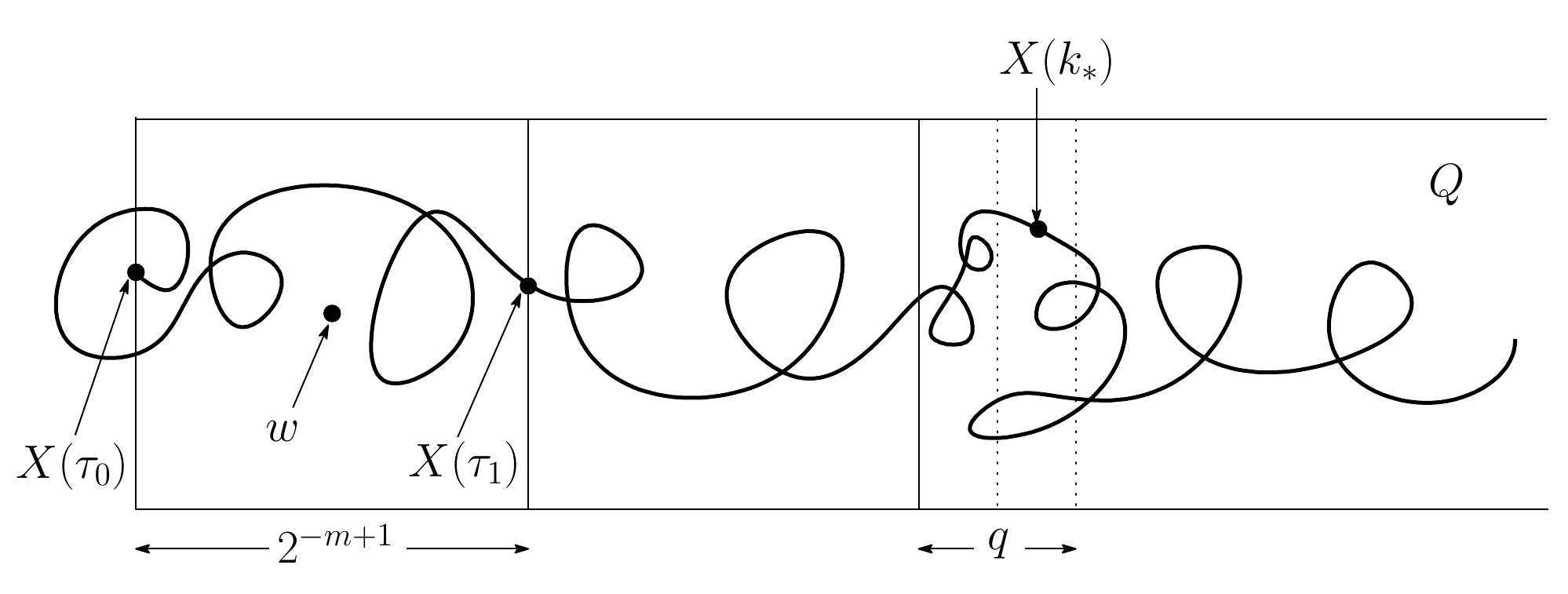}
\caption{Illustration for the events $V_{2}$ and $V_{3}$. $X (k_{\ast})$ is a good cut point.}\label{ponta2}
\end{center}
\end{figure}

Write \begin{equation*}
\kappa = \text{LE} \big( X [0, T^{X} ] \big)
\end{equation*}
for the loop-erasure of $X [0, T^{X} ]$. 

The next lemma shows that on the event 
\begin{equation}\label{eq:Vdef}
    V := V_{1} \cap V_{2} \cap \cdots \cap V_{5},
\end{equation} 
we can control the modulus of continuity of the loop-erasure of $\kappa$. 

Let
\begin{equation*}
\eta_{X} (t) = \text{LE} \kappa ( 2^{\beta n} t ) \ \  \ \ \  0 \le t \le 2^{- \beta n} \cdot {\rm len}( \kappa )
\end{equation*}
be the time rescaled version of $\kappa$. Write 
\begin{equation*}
U^{X} = \inf \big\{ t \ge 0 \ \big| \  | \eta_{X} ( t)   - X (0) | \ge 150 m_{0} \cdot 2^{-m } \big\}.
\end{equation*}
\begin{lem}\label{lem-con-1}
On the event $V$, one has
\begin{equation*}\label{m'}
\sup_{0 \le s < t \le U^{X} } \frac{|\eta_{X} (s) - \eta_{X} (t) |}{|s-t|^{\frac{1}{\beta}}}  \ge c_{\star} \, m_{0}^{1 - \frac{1}{\beta}},
\end{equation*}
for some universal constant $c_{\star} > 0$.
\end{lem}

\begin{proof} 
On the event $V$, it is easy to check that $X [0, k_{\ast} ]  \cap X [ k_{\ast} + 1, T^{X} ] = \emptyset$ where $k_{\ast} \in [\tau_{1}, \tau_{3}]$ is the smallest good cut time.  (All constants used in the definition of $V_{i}$ are chosen to force  $k_{\ast}$ to become a global cut time.) 
Decomposing the path $X [0, T^{X} ]$ at the cut point $X (k_{\ast} )$, we see that $\kappa = \text{LE} \big( X [0, k_{\ast} ] \big) \oplus \text{LE} \big( X [ k_{\ast} , T^{X} ] \big) $. Letting 
\begin{equation*}
s_{1} = \min \big\{ j \ge 0  \ \big| \ \xi (j) \in X [ \tau_{3} , T^{X} ] \big\} \ \ \ \text{ and }  \ \  \  s_{2} =  \max \big\{ \tau_{3} \le j \le T^{X} \ \big| \ X(j) = \xi ( s_{1} ) \big\},
\end{equation*}
(recall that $\xi$ is defined as in \eqref{xi-u1}) it holds that 
\begin{equation*}
\text{LE} \big( X [ k_{\ast} , T^{X} ] \big)  = \xi [0, s_{1} ] \oplus \text{LE} \big( X [ s_{2}, T^{X} ] \big).
\end{equation*}

Write 
\begin{equation*}
U^{\kappa} =  \inf \big\{ j \ge 0 \ \big| \  | \kappa (j)   - z | \ge 150 m_{0} \cdot 2^{-m } \big\},
\end{equation*}
We will prove that on the event $V$
\begin{equation}\label{0906-c-1}
\xi [0, U_{1} ] \subset \kappa [0, U^{\kappa}].
\end{equation}

To prove this, since $X[k_{\ast}, \tau_{2} ] \subset w + H \big[ a_{2}, a_{2 m_{0} + 1} \big] $ and $X [ \tau_{2}, \tau_{3} ] \cap B \big( z, 103 m_{0} \cdot 2^{-m} \big) = \emptyset$ by (ii) and (iii) of \eqref{goodcut}, we see that $\xi [0, U_{1} ]  \subset X [k_{\ast}, \tau_{2} ] \subset  w + H \big[ a_{2}, a_{2 m_{0} + 1} \big] $. In particular, $\xi [0, U_{1} ] \cap X [ \tau_{3}, T^{X} ] = \emptyset$ because   $X [ \tau_{3}, T^{X} ] $ does not return to $B \big( z,   110 m_{0} \cdot 2^{-m } \big)$ by the event $V_{5}$ and $B \big( z,   110 m_{0} \cdot 2^{-m } \big) \cap \Big( w + H \big[ a_{2}, a_{2 m_{0} + 1} \big] \Big) = \emptyset$ by our construction. This implies that $U_{1} < s_{1}$ and $\xi [0, U_{1} ] \subset \kappa$. However, since our construction ensures that $X[0, \tau_{3} ] \subset B \big( z, 145 m_{0} \cdot 2^{-m} \big)$, it follows that  $\text{LE} \big( X [0, k_{\ast} ] \big) \oplus \xi [0, s_{1}] \subset X[0, \tau_{3} ]  \subset B \big( z, 145 m_{0} \cdot 2^{-m} \big)$, and thus
\begin{equation*}
  {\rm len} \Big( \text{LE} \big( X [0, k_{\ast} ] \big) \Big) + s_{1} < U^{\kappa}.
\end{equation*}
This gives $\xi [0, s_{1}]  \subset \kappa [0, U^{\kappa}] $. Combining this with the fact that $U_{1} < s_{1}$, we get \eqref{0906-c-1}.

The event $V_4$ (see \eqref{eq:Videf}) guarantees that $U_{1} \le \widehat{C} m_{0} 2^{- \beta m} 2^{\beta n}$, while $| \xi (0) - \xi (U_{1} ) | = | X (k _{\ast} ) - \xi (U_{1} ) | \ge (2m_{0} - 1) \cdot 2^{-m} - 5 \cdot 2^{-m} \ge m_{0} 2^{-m} $ by the condition (ii) of \eqref{goodcut} for the location of $X (k_{\ast} )$ and our choice of $U_{1}$ as in \eqref{xi-u1} (we also used the fact that $m_{0} \ge 10$ in the last inequality, see \eqref{eq:mdef} for this). Therefore, writing 
\begin{equation}
\eta_{X} (t) = \kappa ( 2^{\beta n} t ) \ \  \ \ \  0 \le t \le 2^{- \beta n} \, {\rm len} ( \kappa ) 
\end{equation}
for the time-rescaled version of $\kappa =  \text{LE} \big( X [0, T^{X} ] \big)$ and letting 
\begin{equation*}
U^{X} = \inf \big\{ t \ge 0 \ \big| \  | \eta_{X} ( t)   - z | \ge 150 m_{0} \cdot 2^{-m } \big\} = 2^{- \beta n} U^{\kappa},
\end{equation*}
we have that 
\begin{equation}\label{m}
\sup_{0 \le s < t \le U^{X}} \frac{|\eta_{X} (s) - \eta_{X} (t) |}{|s-t|^{\frac{1}{\beta}}} \ge \frac{ \big| \xi (0) - \xi ( U_{1} ) 
 \big| }{ \big\{ 2^{ - \beta n} U_{1} \big\}^{\frac{1}{\beta}}  } \ge c_{\star} \, m_{0}^{1 - \frac{1}{\beta}}     \ \ \ \text{ on the event } V,
\end{equation}
where the constant $c_{\star}$ is defined by 
\begin{equation}\label{cstar}
c_{\star} := \widehat{C}^{- \frac{1}{\beta}}.
\end{equation}
This finishes the proof.
\end{proof}

Now the remaining task is therefore to give a lower bound on $P(V)$.

\begin{prop}\label{lem-con-2}
There exists a universal constant $c > 0$ such that 
\begin{equation}\label{0905-1'}
P (V) \ge c \, 2^{ - 4 m_{0}^{2} } .
\end{equation}
\end{prop}

\begin{proof}
We first note that 
\begin{itemize}
\item $V_{1}$ is measurable with respect to $X[0, \tau_{0} ]$;

\item $V_{2}$ is measurable with respect to $X[\tau_{0}, \tau_{1} ]$;

\item $V_{3}$ and $V_{4}$ are measurable with respect to $X[\tau_{1}, \tau_{3} ]$;

\item $V_{5}$ is measurable with respect to $X[\tau_{3}, T^{X} ]$.
\end{itemize}
With this in mind, we will give a lower bound on $P (V)$ by making use of the strong Markov property as follows.

Using the strong Markov property at $\tau_{3}$, we have
\begin{equation}\label{0906-d-1}
P \big( V_{5} \ \big| \ V_{1} \cap V_{2} \cap V_{3} \cap V_{4} \big) \ge \min_{x \in D' } 
P^{x}_X  \Big( X [0, T^{X} ] \cap B \big( z,   110 m_{0} \cdot 2^{-m } \big) = \emptyset \Big),
\end{equation}
where $D' = B \big( w, 41 m_{0} \cdot 2^{-m} \big) \cap \Big( z + \big\{ (y^{1}, y^{2}, y^{3} ) \in \mathbb{R}^{3}  \ \big| \ y^{1} \ge 138 m_{0} \cdot 2^{-m}  \big\} \Big)$, see (iv) of \eqref{goodcut} for the location of $X (\tau_{3})$.
The right-hand side 
of \eqref{0906-d-1} is bounded below by a universal constant $c > 0$, as can be checked similarly to the inequality described in line -8, page 763 of \cite{S} for this. Thus, we have 
\begin{equation}\label{v5-esti}
P (V) \ge c P (V_{1} \cap V_{2} \cap V_{3} \cap V_{4} ).
\end{equation}

We will next deal with the conditional probability of $V_{3} \cap V_{4}$ on the event $V_{1} \cap V_{2}$. We note that, by the Harnack principle (see \eqref{Harnack} for this), it holds that if $\theta $ is a finite path in $2^{-n} \mathbb{Z}^{3}$ satisfying that $x= \theta (0) \in w + G (a_{1} )$ and  $\theta [0, {\rm len} ( \theta ) ] \subset B \big( x, 40 m_{0} \cdot 2^{-m} \big) $ then 
\begin{equation}\label{0906-a}
c P^{x}_{S} \big( S [0, \sigma_{3} ] = \theta \big) \le  P^{x}_{X} \big( X [0, \tau_{3}  ] = \theta \big) \le c^{-1} P^{x}_{S} \big( S [0, \sigma_{3} ] = \theta \big)
\end{equation}
for some universal constant $c > 0$, where $P^{x}_{R}$ stands for the probability law of the random walk $R$ with $R (0)=x$ and $\sigma_{3} : = T^{S}_{x, 40 m_{0} \cdot 2^{-m}} $ stands for the first time that $S$ exits from $B \big( x, 40 m_{0} \cdot 2^{-m} \big)$.   Note that if $x \in w + G (a_{1} )$ and $X (0) = x$ then $\tau_{1} = 0$ and $X [0, \tau_{3} ] = X [\tau_{1}, \tau_{3} ]$.

With \eqref{0906-a} in mind, we can use Proposition \ref{lem-critical} to prove that 
\begin{equation}\label{0906-1}
P \big( V_{3} \cap V_{4} \ \big|  \ V_{1} \cap V_{2} \big)  \ge \min_{x \in w + G (a_{1} )} P^{x} \big( V_{3} \cap V_{4} \big) 
\ge c 2^{- 3 m_{0}^{2}},
\end{equation}
where we applied  Proposition \ref{lem-critical}  to $X [0, \tau_{3} ] = X [\tau_{1} , \tau_{3} ]$ with $X (0) = x \in w + G (a_{1} ) $ in the following way. Combining \eqref{0906-a}, Proposition \ref{lem-critical} and the translation invariance for $S$, if $x \in w + G (a_{1} )$, we have
\begin{align*}
P^{x}_{X}  (V_{3} \cap V_{4} ) \ge c P^{0}_{S} (A^{m}) \ge c 2^{- 3 m_{0}^{2}},
\end{align*}
where the event $A^{m}$ is as defined in \eqref{am}. This gives \eqref{0906-1} since the first inequality of \eqref{0906-1} follows from the strong Markov property. From this and \eqref{v5-esti}, we have
\begin{equation}\label{v43}
P (V) \ge c 2^{- 3 m_{0}^{2}} P (V_{1} \cap V_{2}).
\end{equation}

Proving 
\begin{equation}\label{v2v1}
P ( V_{2}  \ | \ V_{1} ) \ge \min_{x \in w + G (a_{1} )} P^{x} ( V_{2} )  \ge c 
\end{equation}
is easy. In fact, the first inequality of \eqref{v2v1} follows from the strong Markov property, while the second inequality is obtained by comparing $X$ and $S$ through the Harnack principle as in \eqref{0906-a}. From this and \eqref{v43}, the proof of the proposition is completed if we show that 
\begin{equation}\label{vlast}
 P( V_{1} ) \ge c m_{0}^{-2}.
 \end{equation}

We cannot use a similar application of the Harnack principle as in \eqref{0906-a} to compare $X$ and $S$ when $X$ is close to $z = X (0)$. However, \cite[Claim 3.4]{SS} guarantees that 
\begin{equation}\label{vlast-1}
P \Big( X \big( T^{X}_{z, 2 m_{0} \cdot 2^{-m} } \big) \in W \Big) \ge c,
\end{equation}
where we recall that $T^{X}_{x, r}$ stands for the first time that $X$ exits from $B (x, r)$, and we define $W$ as follows. 
\begin{equation*}
W =  B \big( z, 2 m_{0} \cdot 2^{-m}  \big) \cap B \Big( z' , \frac{m_{0}}{2} \cdot 2^{-m} \Big),  \ \ \ \ \ \ z' = z + \big( 2 m_{0} \cdot 2^{-m}, 0, 0 \big).
\end{equation*}
Note that $z'$ lies on the half line $\ell$ satisfying $|z -z'| = 2 m_{0} \cdot 2^{-m}$. 

Once the event of \eqref{vlast-1} occurs,  $X[T^{X}_{z, 2 m_{0} \cdot 2^{-m} }+\cdot]$ then has a reasonable distance from the box $B$ and we can make use of the Harnack principle as in \eqref{0906-a} to show that 
\begin{equation}\label{vlast-2}
\min_{x \in W} P^{x} (V_{1} ) \ge c  m_{0}^{-2}.
\end{equation}
Here we used  the fact that
\begin{equation}
P \Big( S \big( \tau_{ Q (r) } \big) \in F ( \epsilon ) \Big) \ge c \epsilon^{2},
\end{equation}
in \eqref{vlast-2} where 
\begin{itemize}
\item $\tau_{A}$ stands for the first time that $S$ exits from $A$,

\item $Q (r) = [-r, r]^{3}$ is the cube of side length $2 r$ centered at the origin,

\item $y = (r, 0, 0)$, $\epsilon \in  (0, 1)$ and $F (\epsilon) = Q (r) \cap B ( y, \epsilon r )$. 

\end{itemize}
Combining \eqref{vlast-1} and \eqref{vlast-2}, we get \eqref{vlast} and finish the proof.
\end{proof}
 
Finally, as discussed at the beginning of this subsection, we now use an iteration argument similar to those as in   \cite[Proposition 6.6]{BM} and \cite[Proposition 8.11]{S} to show that $\zeta$ is not $1/\beta$-H\"{o}lder continuous in the next proposition, which constitutes the second half of Theorem \ref{main}. Recall the definition of $\eta$ in Section \ref{SCALING}.
\begin{prop}\label{CRITICAL}
 Almost surely, $\eta$ is not $1/\beta$-H\"{o}lder continuous. Namely, with probability one we have 
\begin{equation}\label{eq:notHolder}
\sup_{0 \le s < t \le t_{\eta} } \frac{|\eta (s) - \eta (t) |}{|s-t|^{\frac{1}{\beta}}} = \infty.
\end{equation}
\end{prop}

Note that the event as described in \eqref{eq:notHolder} is measurable in the Borel $\sigma$-algebra generated by the metric $\rho$.

We introduce some notation before giving the proof. Set
\begin{equation}
m_{1} = m^{\frac{1}{2}},
\end{equation}
which is an integer by our assumption. Define 
\begin{equation*}
M_{l} = l \, 2^{- m_{1} } \ \ \ l= 1, 2, \cdots , q_{1}\mbox{, where }q_1=\sup\{q\in\mathbb{Z}^+ | M_q\leq 2/3\}.
\end{equation*}
Note that $q_{1} \asymp 2^{m_{1}}$. Using this, 
we define a sequence of cubes $\{ B^{l} \}_{l=1}^{q_{1}}$ by 
\begin{equation*}
B^{l} = \{ x \in \mathbb{R}^{3} \ | \ \| x \|_{\infty} \le M_{l} \}.
\end{equation*}
Since
\begin{equation}\label{boundarydist}
\text{dist} \big( \partial B^{l}, \partial B^{l+1} \big) = 2^{-m_{1} } = 2^{- m^{\frac{1}{2}} } > 200 m_{0} \cdot 2^{-m},
\end{equation}
there is enough space in each $B^{l+1} \setminus B^{l}$ to insert a translation of $Q$ as in \eqref{0905} in the annulus. See Figure \ref{iteration-f} for illustration. Let 
\begin{equation*}
v_{l} := \tau^{\gamma_{n}}_{B^{l}} = \inf \{ j \ge 0 \ | \ \gamma_{n} (j) \in \partial B^{l} \}
\end{equation*}
be the first time that $\gamma_{n} = \text{LE} \big( S^{(n)} [0, T^{(n)} ] \big) $ exits from $B^{l}$. Recall that $\eta_{n} (t) = \gamma_{n} ( 2^{\beta n} t ) $ stands for the time-rescaled version of $\gamma_{n}$. We set 
\begin{equation*}
v'_{l} = \inf \{ t \ge 0 \ | \ \eta_{n} (t) \in \partial B^{l} \} = 2^{- \beta n} v_{l}  \ \text{and}\    w_{l} = \inf \big\{ t \ge v'_{l} \ \big| \  | \eta_{n} ( t)   - \eta_{n} (v'_{l}) | \ge 150 m_{0} \cdot 2^{-m } \big\}. 
\end{equation*}
Note that $v_{l}' < w_{l}  < v_{l+1}'$ thanks to \eqref{boundarydist}.

\begin{proof}
Define the event $W^{l}$ by
\begin{equation}
W^{l} = \bigg\{ \sup_{ v'_{l} \le s < t \le w_{l}} \frac{|\eta_{n} (s) - \eta_{n} (t) |}{|s-t|^{\frac{1}{\beta}}} <  c_{\star} \, m_{0}^{1 - \frac{1}{\beta}}  \bigg\},
\end{equation}
where $c_{\star} $ is the constant as in Lemma \ref{lem-con-1}.
Applying Lemma \ref{lem-con-1} and Proposition \ref{lem-con-2} to the case that $\lambda = \gamma_{n} [0, v_{l} ]$ and using the domain Markov property, it follows from \eqref{0905-1'} that there exists some universal constant $c > 0$ such that 
\begin{equation*}
P \Big( W^{l} \ \Big| \ \eta_{n} \big[ 0, v'_{l} \big] \Big) \le 1 - c 2^{- 4 m_{0}^{2} }
\end{equation*}
for each $1 \le l \le q_{1}$. Iterating this and recalling that $q_{1} \asymp 2^{m_{1}} = 2^{m^{\frac{1}{2}}}$ and $m_{0} = m^{\frac{1}{10}}$, we have 
\begin{align*}
&P \Big( \bigcap_{l=1}^{q_{1}} W^{l} \Big) = P \Big( W^{q_{1}} \ \Big| \ \bigcap_{l=1}^{q_{1} -1} W^{l} \Big) \, P \Big( \bigcap_{l=1}^{q_{1}-1} W^{l} \Big) \le \Big( 1 - c 2^{- 4 m_{0}^{2} } \Big) \, P \Big( \bigcap_{l=1}^{q_{1}-1} W^{l} \Big)  \\
& \le \cdots \le 
 \Big\{ 1 - c 2^{- 4 m_{0}^{2} } \Big\}^{ c 2^{m^{\frac{1}{2}}} } \le C \exp \Big\{ - c' \, 2^{ c' m^{\frac{1}{2}} } \Big\},
\end{align*}
where we used the fact that the event $\bigcap_{l=1}^{q_{1} -1} W^{l}$ is measurable with respect to $\eta_{n} [0, v_{q_{1}}']$ in the second inequality thanks to $w_{q_{1}-1} < v_{q_{1}}'$.
Therefore, if we write 
\begin{equation}
A_{m, n} =  \bigg\{ \sup_{ 0 \le s < t \le t_{\eta_{n}} } \frac{|\eta_{n} (s) - \eta_{n} (t) |}{|s-t|^{\frac{1}{\beta}}} \ge  c_{\star} \, m_{0}^{1 - \frac{1}{\beta}}  \bigg\},
\end{equation}
it holds that 
\begin{equation}\label{p}
P \big( A_{m, n} \big) \ge 1 - C \exp \Big\{ - c' \, 2^{ c' m^{\frac{1}{2}} } \Big\}.
\end{equation}
Since $\eta_{n}$ converges to $\eta$ with respect to the metric $\rho$ (see Section \ref{metric} for $\rho$), it follows from \eqref{p} that 
\begin{equation*}
P \bigg[   \sup_{ 0 \le s < t \le t_{\eta}} \frac{|\eta (s) - \eta (t) |}{|s-t|^{\frac{1}{\beta}}} \ge  \frac{c_{\star}}{2}  \cdot m_{0}^{1 - \frac{1}{\beta}}   \bigg] \ge 1 - C \exp \Big\{ - c' \, 2^{ c' m^{\frac{1}{2}} } \Big\}.
\end{equation*}
Using the Borel-Cantelli lemma, it holds that with probability one there exists $m'$ such that
\begin{equation*}
 \sup_{ 0 \le s < t \le t_{\eta}} \frac{|\eta (s) - \eta (t) |}{|s-t|^{\frac{1}{\beta}}} \ge  \frac{c_{\star}}{2}  \cdot m^{\frac{1}{10} - \frac{1}{10 \beta}}  \  \ \text{ for all } \ m \ge m'.
 \end{equation*}
Since $\beta \in (1, \frac{5}{3} ] $, letting $m \to \infty$, the proof is completed. 
\end{proof}

\begin{acks}
The authors thank two anonymous referees for their careful reading and numerous helpful comments on an earlier draft of this work. XL's research is supported by the National Key R\&D Program of China (No.\ 2020YFA0712900 and No.\ 2021YFA1002700) and NSFC (No.\ 12071012). 
DS would like to thank David Croydon for his kind explanation of results and techniques in \cite{BCK}. The proof of the second claim of the main theorem is inspired by the discussion.
DS is supported by a JSPS Grant-in-Aid for Early-Career Scientists, 18K13425 and JSPS KAKENHI Grant Number 17H02849 and 18H01123. 
\end{acks}

\end{document}